\documentclass[11pt]{article}
\usepackage{amsfonts,latexsym,amssymb,amsthm,amsmath,graphicx,cases,comment}
\usepackage{authblk}
\usepackage{paralist}
\usepackage{graphics} %% add this and next lines if pictures should be in esp format
\usepackage{epsfig} %For pictures: screened artwork should be set up with an 85 or 100 line screen
\usepackage{epstopdf}%This is to transfer .eps figure to .pdf figure; please compile your paper using PDFLeTex or PDFTeXify.
\usepackage{ulem}
\usepackage[titletoc,title]{appendix}
\usepackage{empheq}
\usepackage{cancel}
\usepackage{tikz}
\usetikzlibrary{arrows,shapes}
\usetikzlibrary{decorations.pathmorphing,decorations.pathreplacing}
\usetikzlibrary{calc,patterns,angles,quotes}
\usepackage{lineno}
\usepackage[colorlinks=true]{hyperref}
\hypersetup{urlcolor=blue, citecolor=red}
\makeatletter
\renewcommand{\subsection}{%
	\@startsection{subsection}
	{2}
	{\z@}
	{-21dd plus-8pt minus-4pt}
	{10.5dd}
	{\normalsize\bfseries\boldmath}%
}
\makeatother

\newtheorem{theorem}{Theorem}[section]
\newtheorem{thm}{Theorem}

\newtheorem{prop}[thm]{Proposition}

\theoremstyle{definition}
\newtheorem{definition}[theorem]{Definition}
\newtheorem{rmk}{Remark}

\newcommand{\be}{\begin{equation}}
\newcommand{\ee}{\end{equation}}
\newcommand{\bsubeq}{\begin{subequations}}
	\newcommand{\esubeq}{\end{subequations}}
\renewcommand{\div}{\text{div}}
\newcommand{\ds}{\displaystyle}

\newcommand{\calL}{{\mathcal{L}}}

\newcommand{\calF}{{\mathcal{F}}}
\newcommand{\calB}{{\mathcal{B}}}
\newcommand{\calD}{{\mathcal{D}}}

\newcommand{\calA}{{\mathcal{A}}}

\newcommand{\calC}{{\mathcal{C}}}

\newcommand{\BA}{\mathbb{A}}

\newcommand{\BR}{\mathbb{R}}

\newcommand{\wti}{\widetilde}
\newcommand{\what}{\widehat}

\newcommand{\bpm}{\begin{pmatrix}}
	\newcommand{\epm}{\end{pmatrix}}

\newcommand{\bbm}{\begin{bmatrix}}
	\newcommand{\ebm}{\end{bmatrix}}

\numberwithin{equation}{section}
\numberwithin{thm}{section}
\numberwithin{rmk}{section}
\newtheorem{clr}[thm]{Corollary}

\newcommand{\lso}{\bs{L}^{q}_{\sigma}(\Omega)}
\newcommand{\lqo}{L^{q}(\Omega)}
\newcommand{\lo}[1]{\bs{L}^{#1}_{\sigma}(\Omega)}

\newcommand{\Bso}{\bs{B}^{2-\rfrac{2}{p}}_{q,p}(\Omega)}

\newcommand{\Bto}{\widetilde{\bs{B}}^{2-\rfrac{2}{p}}_{q,p}(\Omega)}

\newcommand{\xipqs}{\bs{X}^{\infty}_{p,q,\sigma}}

\newcommand\rfrac[2]{{}^{#1}\!/_{#2}}

\newcommand{\lqaq}{\big( \lso, \calD(A_q) \big)_{1-\frac{1}{p},p}}

\newcommand{\lqafq}{\big( \lso, \calD(\BA_{_{F,q}}) \big)_{1-\frac{1}{p},p}}

\newcommand{\norm}[1]{\left\lVert#1\right\rVert}
\newcommand{\abs}[1]{\left\lvert#1\right\rvert}

\newcommand{\lplqs}{L^p \big( 0,\infty; \lso \big)}

\newcommand{\ip}[2]{\left\langle #1, #2 \right\rangle}

\newcommand{\bs}[1]{\boldsymbol{#1}}
\newcommand{\bls}{\bs{L}^q_{\sigma}(\Omega)}

\newcommand{\nin}{\noindent}

\newcommand{\thistheoremname}{}

\newcommand{\Hi}{(\textbf{H.1})}
\newcommand{\Hii}{(\textbf{H.2})}
\newcommand{\Hiii}{(\textbf{H.3})}
\newcommand{\Hiv}{(\textbf{H.4})}
\newcommand{\Hv}{(\textbf{H.5})}

\newtheorem*{genericthm*}{\thistheoremname}
\newenvironment{namedthm*}[1]
{\renewcommand{\thistheoremname}{#1}%
	\begin{genericthm*}}
	{\end{genericthm*}}

\begin{document}
	\title{Maximal $L^p$-regularity for an abstract evolution equation with applications to closed-loop boundary feedback control problems}
	\author[1,2]{Irena Lasiecka}
	\author[3]{Buddhika Priyasad}
	\author[1]{Roberto Triggiani}
	\affil[1]{Department of Mathematical Sciences, University of Memphis, Memphis, TN 38152 USA}
	\affil[2]{IBS, Polish Academy of Sciences, Warsaw, Poland}
	\affil[3]{621 Institute for Mathematics and Scientific Computing, 8010 Graz, Heinrichstra\ss e 36, Austria}
	\setcounter{Maxaffil}{0}
	\renewcommand\Affilfont{\itshape\small}
	\date{}
%	\authorrunning{Lasiecka et. al}
%	\titlerunning{Uniform boundary stabilization of Navier-Stokes equation}
	\maketitle

	%The abstract of your paper
	\begin{abstract}
	\noindent In this paper we present an abstract maximal  $L^p$-regularity result up to $T = \infty$, that is tuned to capture (linear) Partial Differential Equations of parabolic type, defined on a bounded domain and subject to finite dimensional, stabilizing, feedback controls acting on (a portion of) the boundary. Illustrations include, beside a more classical	boundary parabolic example, two more recent settings: (i) the $3d$-Navier-Stokes equations with finite dimensional, localized, boundary tangential feedback stabilizing controls as well as Boussinesq systems with finite dimensional, localized, feedback, stabilizing, Dirichlet	boundary control for the thermal equation.
	\end{abstract}
	
%	\tableofcontents
	
	\section{Introduction and statement of main result.}\label{Sec-1}
	
	Though written at the outset at the abstract level for an abstract linear model, the present paper is actually motivated by, and ultimately directed to, nonlinear Partial Differential Equations (PDEs) of parabolic type, defined on a bounded domain and subject to finite dimensional, stabilizing, \textit{feedback controls acting on (a portion of) the boundary}. A key preliminary goal is to establish uniform stabilization of the corresponding \textit{linearized boundary-based}, feedback, closed-loop problem. The extra property to be established is that such boundary-based feedback linearized system possesses the maximal  $L^p$-regularity property up to $T = \infty$ in the natural functional setting, where uniform stabilization is achieved. Maximal  $L^p$-regularity up to $T = \infty$ is then critically used to provide a novel, much streamlined, improved treatment of the consequent nonlinear analysis of well-posedness and uniform stabilization of the nonlinear parabolic problem in the vicinity of an unstable equilibrium solution. See \cite{LPT.2} for the Navier-Stokes equations and \cite{LPT.4} for the Boussinesq system, to be compared to prior treatments such as in \cite{B:2011}, \cite{B:2018}, \cite{BT:2004}, \cite{BLT1:2006}, \cite{BLT2:2007}, \cite{BLT3:2006}, \cite{LT2:2015}. Our driving motivating illustration is the $3d$ Navier-Stokes equations of Section \ref{Sec-5}. Here, the functional setting where uniform stabilization with a finite dimensional, localized, boundary, even tangential, control is achieved in full generality cannot be a Hilbert-Sobolev space. In fact, in studying local well-posedness and uniform stabilization near an unstable equilibrium solution, handling the N-S nonlinearity requires a sufficiently high topological level as to impose compatibility conditions between the initial conditions and the boundary-based control. In truth, whether it was possible at all to achieve uniform stabilization of the Navier-Stokes equations in the vicinity of an unstable equilibrium solution by virtue of a localized boundary-based feedback control that is finite dimensional also for $d = 3$ was an open problem that was solved in the affirmative in the recent paper \cite{LPT.2}. It required a suitable Besov space setting, with tight indices, based on $L^q(\Omega)$, $q > d$ and `close' to $L^3(\Omega)$, which possesses two features: (i) a topological level high enough to be able to handle the $3d$ non-linearity; (ii) without recognizing compatibility conditions. Such Besov setting then replaces the Hilbert-Sobolev setting that was traditionally used in the literature on parabolic stabilization of fluids over many years. Stabilization of the Navier-Stokes equations was pioneered by A. Fursikov \cite{F.1}, \cite{F.2}, \cite{F.3}.
	
	While it has been known for many years that analyticity of the s.c. semigroup and maximal regularity are equivalent properties in the Hilbert setting \cite{DS:1964}, in the Banach setting maximal regularity implies, but need not be implied by, maximal regularity \cite{Dore:2000}, \cite{KW:2004}.
	
	Maximal regularity at the abstract functional analytic level, as well as maximal regularity of (linear) parabolic problems on bounded (or even unbounded) domains is of course a much worked out topic over many years; however, in the latter case of a bounded domain, typically with \uline{homogeneous} boundary conditions. There is a vast literature on this topic, that covers non only classical parabolic operators but also Navier-Stokes-based operators such as the Stokes operator, see \cite{DV:1991}, \cite{Dore:1991}, \cite{Dore:2000}, \cite{DS:1964}, \cite{DPV:2002}, \cite{DPG:1975}, \cite{DPG:1984}, \cite{Gi:1981}, \cite{G:1969}, \cite{VAS:1968}, \cite{KW:2001}, \cite{KW:2004}, \cite{V:1985}, \cite{We:2001}, \cite{PS:2016}, to name a few.
	
	In this paper, our focus is instead on linear parabolic problems with \textit{boundary-based} stabilizing feedback control of finite dimension.
	
	Our abstract maximal $L^p$-regularity theorem up to $T = \infty$ of Section \ref{Sec-1} is tuned to capture the $3d$ Navier-Stokes linearized illustration of Section \ref{Sec-5}, mentioned above. In fact, there are critical genuine intrinsic properties pertinent to such Navier-Stokes-illustration in the $L^q(\Omega)$-setting, $q > d$ that are extracted  and elevated to become abstract assumptions of the theorem of Section \ref{Sec-1}.
	
	To ease the transition on the applications, we provide in Section \ref{Sec-4} a more classical illustration of parabolic boundary stabilization, that was studied in the Hilbert setting $L^2(\Omega)$ in the early 80s \cite{LT:1983.1}, \cite{LT:1983.2}, \cite{RT:1975}, \cite{RT:1980}, \cite{RT:1980:2} and which is here re-presented in the $L^q$ setting, to conclude now with maximal $L^p$-regularity up to $T = \infty$ in the stabilized case. Section \ref{Sec-6} then goes on by providing an additional illustration of our abstract theorem of Section \ref{Sec-1}, as applied to a linearized Boussinesq system, coupling the Navier-Stokes equations with a thermal equation, where a Dirichlet boundary stabilizing finite dimensional feedback localized control acts on the heat component \cite{LPT.4}. Again, ultimately, well-posedness and uniform stabilization is achieved in a Besov space setting, by virtue of the asserted maximal $L^p$-regularity of the linearized system.	
	
	\subsection{Standing assumptions.}
	
	\nin We introduce the following assumptions.\\
	
	\nin (\textbf{H.1}) Let $Y$ be a Banach space which, moreover, is a UMD-space \cite[p 75]{KW:2004}; hence reflexive \cite[Theorem 4.3.3, p306]{HNVW:2016}.\\
	
	\nin (\textbf{H.2}) Let $-A: Y \supset \calD(A) \longrightarrow Y$ be the generator of a s.c. bounded analytic semigroup $\ds e^{-At}$ on $Y$, $t \geq 0$. Accordingly, the fractional powers $A^{\theta}, 0 < \theta < 1$, of $A$ are well-defined, possibly after a translation.\\
	
	\nin (\textbf{H.3}) Let $-A^* \supset Y^* \supset \calD(A^*) \longrightarrow Y^*$ have maximal $L^p$-regularity on $Y^*$ up to $T$: $\ds -A^* \in MReg \left( L^p(0,T; Y^*) \right)$ (so that, a-fortiori, $-A^*$ is the generator of a s.c. bounded analytic semigroup $\ds e^{-A^*t}$ on $Y^*$; as implied by \Hii \ via the reflexivity of $Y$ in \Hi).\\ 
	
	\nin (\textbf{H.4}) Let $U$ be another Banach space and let $G$ be a (``Green") operator satisfying
	\begin{equation}\label{1.1}
		G: \text{continuous } U \longrightarrow \calD(A^{\gamma}) \subset Y, \text{ or } A^{\gamma}G \in \calL(U,Y)
	\end{equation}
	for some constant $\gamma,\ 0 < \gamma < 1$.\\
	
	\nin (\textbf{H.5}) Consider the following three operators $\ds A_o, \calA, F$
	\begin{subequations}\label{1.2}
		\begin{eqnarray}
			A_o: Y \supset \calD(A_o) = \calD(A^{1-\varepsilon}) \longrightarrow Y, \ \varepsilon > 0, \label{1.2a}\\
			\text{with } Y^* \supset \calD(A_o^*) = \calD \big( A^{*^{1-\varepsilon}} \big) \longrightarrow Y^* \label{1.2b}
		\end{eqnarray}
	\end{subequations}
	so that, by the closed graph theorem, $\ds A_o A^{-(1-\varepsilon)} \in \calL(Y)$ is a bounded operator on $Y$ and $\ds A^*_o A^{*^{-(1-\varepsilon)}} \in \calL(Y^*)$ is a bounded operator on $Y^*$;
	\begin{eqnarray}
	\calA = -A + A_o: \ Y \supset \calD(\calA) = \calD(A) \longrightarrow Y; \label{1.3}\\
	F \in \calL(Y,U), \ F \text{ (stands for ``feedback")}. \label{1.4}
	\end{eqnarray}
	
	\nin Since the perturbation $A_o$ of $-A$ in \eqref{1.3} is $\ds A^{1-\varepsilon}$-bounded and $-A$ is a s.c. analytic semigroup generator, it follows \cite[Corollary 2.4, p 81]{P:1983} that
	
	\begin{equation}\label{1.5}
		\calA \text{ is the generator of a strongly continuous analytic semigroup } e^{\calA t} \text{ on } Y, t > 0.
	\end{equation}
	The focus of our main interest in this paper is the operator 
	\begin{subequations}\label{1.6}
		\begin{align}
			A_{_F} &= \calA (I - GF): Y \supset \calD(A_F) \longrightarrow Y \label{1.6a}\\
			\begin{picture}(0,0)
			\put(-15,10){$\left\{\rule{0pt}{20pt}\right.$}\end{picture}
			\calD(A_{_F}) &= \left \{ x \in Y: \ (I - GF)x \in \calD(\calA) \right \}. \label{1.6b}
		\end{align}
	\end{subequations}
	\begin{rmk}\label{rmk-1.1}
		\begin{enumerate}[a)]
			\item %The class of UMD-spaces includes $L^q(\Omega)$ spaces, $1 < q < \infty$. \cite[p 75]{KW:2004}, where $\Omega$ is an open bounded domain in $\BR^d$, with sufficiently smooth boundary $\Gamma = \partial \Omega$. We shall write the theoretical part of this paper (Sections \ref{Sec-1} \ref{Sec-2}, \ref{Sec-3}) an abstract level. However, there is one step in the proof of the main result, Theorem \ref{Thm-1.2}, which seems to require that $Y$ be specialized to $L^q(\Omega)$: it occurs when we appeal to a result in \cite[Corollary 2.11, p 90]{KW:2004} on the complete duality for R-boundedness in $L^q(\Omega)$. It will be explicitly noted when $Y$ is restricted to $L^q(\Omega)$, so that $\ds Y^* = L^{q'}(\Omega), \rfrac{1}{q} + \rfrac{1}{q'} = 1$.
			We quote from \cite[p 75]{KW:2004}: ``\textit{All subspaces and quotient spaces of $L^q(\Omega),\ 1 < q < \infty$ have the UMD property \textit{but $L^1(\Omega)$} or spaces of continuous functions $C(K)$ do not. As a rule of thumb, we can say that Sobolev spaces, Hardy spaces and other well-known spaces of analysis are UMD if they are reflexive}".
			
			\item In applications to PDE closed-loop systems, $\Omega$ is an open bounded domain in $\BR^d$ with sufficiently smooth boundary $\Gamma = \partial \Omega$, while the feedback control acts on the boundary $\Gamma = \partial \Omega$ of $\Omega$. Then the space $U$ will be based on $\Gamma$, say $U = L^q(\Gamma)$, possibly subject to further conditions. The operator $\calA$ has compact resolvent.
		\end{enumerate}
	\end{rmk}

	\subsection{The dynamical model generated by $A_{_F}$: Main results.}
	
	\begin{prop}\label{prop-1.1}
		Under the given assumptions \Hi, \Hii, \Hiv, \Hv \ the operator $A_F$ in \eqref{1.6} generates a s.c. analytic semigroup $\ds e^{A_{_F}t}$ on $Y, \ t \geq 0$. If moreover $\calA$ has compact resolvent, then likewise the resolvent $R(\lambda, A_{_F})$ is compact on $Y$ and so the semigroup $\ds e^{A_F t}$ is compact as well for all $t > 0$. \cite[Thm 3.3, p 48]{P:1983}
	\end{prop}
	\nin A proof is given in Section \ref{Sec-2}. On the basis of Proposition \ref{prop-1.1}, we consider the following abstract dynamical system on the space $Y$:
	\begin{subequations}\label{1.7}
		\begin{align}
			\frac{dy}{dt} &= \calA (I - GF)y + f \equiv A_F + f, \quad y(0) = y_0 \in Y\\
			y(t) &= e^{A_{_F} t}y_0 + \int_{0}^{t} e^{A_{_F} (t - s)} f(s) ds
		\end{align}	
	\end{subequations} 
	with the forcing term $f$ specified below. Equation \eqref{1.7} serves as an abstract model of the Partial Differential Equations of parabolic type, written in a closed loop form, with feedback controls acting on the boundary $\Gamma$ of the smooth bounded domain $\Omega \in \BR^d$ in Remark \ref{rmk-1.1}. This will be illustrated in subsequent sections. The goal of the present paper is to establish the following result on the maximal $L^p$-regularity on $Y$ of the feedback operator $A_F$.
	
	\begin{thm}\label{Thm-1.2}
		Assume \Hi-\Hv. With reference to the dynamics \eqref{1.7}, let $y_0 = 0$
		\begin{enumerate}[a)]
			\item Then the map
			\begin{subequations}\label{1.8}
				\begin{align}
					f &\mapsto (Lf)(t) = \int_{0}^{t} e^{A_{_F} (t - s)} f(s) ds:\\
					\text{ continuous } L^p(0,T;Y) &\longrightarrow X^T_p \equiv L^p(0,T;\calD(A_{_F})) \cap W^{1,p}(0,T;Y), \ 1 < p < \infty
				\end{align}
			\end{subequations}
		\nin so that there is a constant $C = C_{p,T} > 0$ such that
		\begin{equation}\label{1.9}
			\norm{y_t}_{L^p(0,T;Y)} + \norm{A_{_F} y}_{L^p(0,T;Y)} \leq C \norm{f}_{L^p(0,T;Y)}.
		\end{equation}
		\nin In short: the operator $A_F$ has maximal $L^p$-regularity on $Y$ up to $T < \infty$. We express this symbolically, using the notation of \cite{Dore:2000}
		\begin{equation}\label{1.10}
			A_{_F} \in MReg (L^p(0,T;Y)).
		\end{equation}
		\item Assume further that the s.c. analytic semigroup $e^{A_{_F}t}$ of Proposition \ref{prop-1.1} is uniformly stable on $Y$. Then, the above results \eqref{1.8}, \eqref{1.9} hold true with $T = \infty$, so that \cite[Theorem 5.2, p307]{Dore:2000}
		\begin{equation}\label{1.11}
			A_{_F} \in MReg (L^p(0,\infty;Y)).
		\end{equation}
		\item Suppose there exists a bounded operator $B \in \calL(Y)$ such that the s.c. analytic semigroup $\ds e^{\BA_{_{F}}t}$ generated via Proposition \ref{prop-1.1} by
		\begin{equation}\label{1.12}
			\BA_{_{F}} = A_{_F} + B = \calA (I - GF) + B
		\end{equation}
		\nin is uniformly stable in $Y$. Then 
		\begin{equation}\label{1.13}
			\BA_{_F} \in MReg (L^p(0,\infty;Y)).
		\end{equation}
		\end{enumerate}
	\end{thm} 
	\begin{rmk}\label{rmk-1.2}
		Case b) occurs in the case of uniform stabilization of the closed-loop linearized Navier-Stokes equations with a feedback control pair $\{ \bs{v}, \bs{u} \}$, with boundary feedback control $\bs{v}$ acting on an arbitrary small connected portion $\wti{\Gamma}$ of the boundary $\Gamma = \partial \Omega$ of the bounded domain $\Omega$, and interior control $\bs{u}$ acting tangentially (parallel to $\wti{\Gamma}$) on an arbitrary interior collar $\omega$ supported by $\wti{\Gamma}$. Fig 2, Section \ref{Sec-5}. The operator $F$ is the feedback operator for $\bs{v}$, the operator $B$ is the feedback operator for $\bs{u}$. The control $\bs{u}$ \textit{cannot be dispensed with} to obtain maximal $L^p$-regularity up to $T = \infty$: the presence of the additional bounded operator $B \in \calL(Y)$ is \textit{critical} to achieve such uniform stabilization. With $B = 0$ one obtains maximal $L^p$-regularity only up to any $T < \infty$. All this is to be discussed in Section \ref{Sec-5}. In other cases (Sections \ref{Sec-4} and \ref{Sec-6}), we can take $B = 0$. %(ii) If in \Hiii \ we do not assume boundedness of the semigroup $\ds e^{-A^*t}$, then $L^p$ maximal regularity holds true for all $T < \infty$.
	\end{rmk}

	\section{Proof of Proposition \ref{prop-1.1}.} \label{Sec-2}
	The short proof, patterned after \cite[p 151]{L-T.1}, is inserted here for completeness. It uses two key ingredients: the classical perturbation theory of the resolvent $R(\lambda, A_{_F})$ in terms of $R(\lambda, \calA)$ \cite[p 80]{P:1983}; and assumption \eqref{1.1} on $G$, in addition to \eqref{1.4} for $F$. Both statements: (i) that $A_{_F}$ generates a s.c. analytic semigroup $\ds e^{A_{_F}t}$ on $Y,\ t > 0$, and (ii) that the resolvent $R(\lambda, A_{_F})$ is compact on $Y$ rely on the first ingredient. The classical perturbation formula \cite[p 80]{P:1983} written for $A_{_F}$ in \eqref{1.6} is:
	\begin{equation}\label{2.1}
		R(\lambda, A_{_F}) = \left[ I + R(\lambda, \calA)\calA GF \right]^{-1} R(\lambda, \calA),
	\end{equation}
	\nin at least for $\lambda \in \rho(A)$, with $Re \ \lambda > $ some $\rho_0 > 0$. Next, property $\ds A^{\gamma}G \in \calL(Y)$ for some $0 < \gamma < 1$ as in \eqref{1.1} yields $\ds \hat{\calA}^{\gamma} G = (kI - \calA)^{\gamma} G \in \calL(U,Y)$ for some $k > 0$ suitably large, as $\ds \calD(\calA) = \calD(A)$ by \eqref{1.3}, where $\ds -\hat{\calA} = \calA - kI$ generates a s.c. analytic semigroup $\ds e^{\hat{\calA}t}$ on $Y$ by \eqref{1.5}. Moreover, $\ds \hat{\calA}^{\gamma}GF \in \calL(Y)$ by \eqref{1.4}. Accordingly a well-known formula \cite[Eq (5.15), p 115]{Krein} gives 
	\begin{align}
		\norm{R(\lambda, \hat{\calA}) \hat{\calA}GF }_{\calL(Y)} = \norm{R(\lambda, \hat{\calA}) \hat{\calA}^{1 - \gamma}(\hat{\calA}^{\gamma}GF)}_{\calL(Y)} &\leq C_{\gamma}\norm{R(\lambda, \hat{\calA}) \hat{\calA}^{1 - \gamma}}_{\calL(Y)} \label{2.2}\\
		&\leq \frac{\tilde{C}_{\gamma}}{\abs{\lambda}^{\gamma}} \longrightarrow 0 \text{ as } \abs{\lambda} \longrightarrow \infty, \lambda \in \rho(\hat{\calA}) \label{2.3}
	\end{align}
	\nin Then by \eqref{2.1} and \eqref{2.3}, we obtain
	\begin{equation}\label{2.4}
		\norm{R(\lambda, A_{_F})}_{\calL(Y)} \leq C_{\gamma, \rho_0} \norm{R(\lambda, \calA)}_{\calL(Y)}, \quad \forall \lambda,\ Re \ \lambda > \text{ some } \rho_0 > 0.
	\end{equation}
	\nin Thus, via \eqref{2.4} the properties of $R(\lambda, \calA)$ [generation by $A$ of a s.c. analytic semigroup on $Y$ by \eqref{1.5} and compactness transfer into corresponding properties for $R(\lambda, A_F)$] \cite[Lemma 4.2.3, p 185]{Fat.1}. Finally, compactness of the resolvent and analyticity of the semigroup a-fortiori imply compactness of the semigroup for all $t \geq 0$ \cite[Thm 3.3 p48]{P:1983}. Proposition \ref{prop-1.1} is proved.
	
	\section{Proof of Theorem \ref{Thm-1.2}.}\label{Sec-3}
	
	\nin \textit{Part a)} \uline{Step 1:} With $F \in \calL(Y,U)$ by \eqref{1.4} and $G$ satisfying \eqref{1.1}, the intrinsic presence of the operator $GF$ as a right factor in the expression of $A_{_F}$ in \eqref{1.6a} makes such expression not directly suitable for deducing its maximal $L^p$-regularity on $Y$, as it would leave the power $A^{1-\gamma}$ on the LHS unaccounted for on $Y$. Accordingly, we find it convenient to consider instead the more amenable adjoint/dual operator.	
	\begin{subequations}\label{3.1}
		\begin{eqnarray}
			A^*_{_F} = (I - GF)^*\calA^* = -(I - GF)^*A^* + (I - GF)^*A_o^* \label{3.1a}\\
			Y^* \supset \calD(A^*_{_F}) = \calD(\calA^*) = \calD(A^*) \longrightarrow Y^* \label{3.1b}
		\end{eqnarray}
	\end{subequations}
	\nin via \eqref{1.3}, since $GF \in \calL(Y)$ and $(I - GF)^* \in \calL(Y^*)$. We rewrite $A_F^*$ in \eqref{3.1a} as
	\begin{equation}\label{3.2}
		A^*_{_F} = -A^* + [F^*G^*A^{*^{\gamma}}]A^{*^{1-\gamma}} + (I -  GF)^*(A^{-(1-\varepsilon)}A_o)^*A^{*^{1-\varepsilon}}
	\end{equation}
	\nin whereby the adjoint of the right factor $(I-GF)$ in \eqref{1.6a} becomes now a left factor $(I-GF)^*$ in \eqref{3.2}. In obtaining in \eqref{3.1a} the form of $A_F^*$ from that of $A_F$ in \eqref{1.6a}, we have used \cite[p 14]{Fat.1} that $(I - GF) \in \calL(Y)$. Moreover, we have also used $\ds A_o = A^{1-\varepsilon} (A^{-(1-\varepsilon)}A_o)$, hence $\ds A_o^* = (A^{-(1-\varepsilon)}A_o)^* A^{*^{1-\varepsilon}}$, with $(A^{-(1-\varepsilon)}A_o)^* \in \calL(Y^*)$ by \eqref{1.2b}.\\
	
	\nin \uline{Step 2:} By duality on Proposition \ref{prop-1.1} on the reflexive Banach space $Y$, the operator $A_F^*$ in \eqref{3.1} generates a s.c. analytic semigroup $\ds e^{A^*_F t}$ on $Y^*$.\\
	
	\nin \uline{Step 3:}
	
	\begin{prop}\label{prop-3.1}
		For the generator $A^*_F$ in \eqref{3.1} of the s.c. analytic semigroup $e^{A^*_F t}$ on $Y^*$, we have
		\begin{equation}\label{3.3}
			A^*_{_F} \in MReg(L^p(0,T;Y^*)), \ 0 < T < \infty.
		\end{equation}
	\end{prop}
	
	\begin{proof}
		The proof is based on a perturbation argument. With $\ds [A^{\gamma} GF]^* = F^* G^* A^{*^{\gamma}} \in \calL(Y^*)$ by \Hiv =\eqref{1.1} and \eqref{1.4}, rewrite \eqref{3.2} as:
		\begin{align}
			A_F^* &= - A^* + \Pi \label{3.4}\\
			\Pi &= [F^* G^* A^{*^{\gamma}}]A^{*^{1 - \gamma}} + [(I - GF)^* (A^{-(1 - \varepsilon)}A_o)^*]A^{*^{1 - \varepsilon}}. \label{3.5}
		\end{align}
		\nin In \eqref{3.5}, both terms in the square brackets $[ \quad ]$ are bounded in $Y^*$ by assumption \Hiv \ = \eqref{1.1} and \Hv. The following estimates then hold true:
		\begin{align}
			(i) \ &\norm{[F^* G^* A^{*^{\gamma}}]A^{*^{1 - \gamma}}x}_{Y^*} \leq C \norm{A^{*^{1 - \gamma}}x}_{Y^*}, \ \forall x \in \calD \big(A^{*^{1 - \gamma}} \big) \label{3.6}\\
			(ii) \ &\norm{[(I - GF)^* (A^{-(1 - \varepsilon)}A_o)^*]A^{*^{1 - \varepsilon}}x}_{Y^*} \leq C \norm{A^{*^{1 - \varepsilon}}x} _{Y^*},  \ \forall x \in \calD \big(A^{*^{1 - \varepsilon}} \big). \label{3.7}
		\end{align} 
		\nin Hence,  by \eqref{3.6}, \eqref{3.7} the perturbation $\Pi$ in \eqref{3.5} satisfies
		\begin{equation}
			\norm{\Pi x}_{Y^*} \leq C \norm{A^{*^{\theta_0}}x}_{Y^*}, \ x \in \calD \big(A^{*^{\theta_0}} \big), \ \theta_0 = \max \{ 1 - \varepsilon, 1 - \gamma\} < 1. \label{3.8}
		\end{equation} 
		We are now in a position to draw some consequences from \eqref{3.4}, \eqref{3.8}:
		\begin{enumerate}[(a)]
			\item The perturbation $\Pi$ is $\ds A^{*^{\theta_0}}$-bounded on $Y^*$, $0 < \theta_0 < 1$.
			\item On the other hand, by \Hiii,  we have $\ds A^* \in MReg(L^p(0,T;Y^*))$. 
		\end{enumerate}
	\nin Then via \eqref{3.4}, properties (a), (b) imply via \cite[Theorem 6.2, p 311]{Dore:2000} or \cite[Remark 1i, p 426 for $\beta = 1$]{KW:2001} that $\ds A^*_{_F} \in MReg(L^p(0,T;Y^*))$ and Proposition \ref{prop-3.1} is proved.
	\end{proof}
	\nin \uline{Step 4:} We now prove Theorem \ref{Thm-1.2} that $\ds A_{_F} \in MReg(L^p(0,T;Y)$ as claimed in \eqref{1.10}. To this end, we invoke the fundamental result of L. Weis \cite[Theorem 1.11, p 76]{KW:2004}, \cite[Theorem, p 198]{We:2001}. Since by Proposition \ref{prop-1.1}, $A_{_F}$ generates a s.c. analytic semigroup $\ds e^{A_{_F}t}$ on the UMD-space $Y$ which modulo a translation (innocuous for the present argument), we may take to be bounded. Then the sought after property that $\ds A_{_F} \in MReg(L^p(0,T;Y))$ is equivalent to the property that the family $\tau \in \calL(Y)$
	\begin{equation}\label{3.9}
		\tau = \left\{ tR(it,A_{_F}), \ t \in \BR\backslash \{0\} \right\} \text{ be } R\text{-bounded},
	\end{equation}
	\nin where $\ds R( \cdot ,A_{_F})$ denotes the resolvent operator of $\ds A_{_F}$. However, in our present UMD setting for $Y$, the family $\tau$ in \eqref{3.9} is $R$-bounded if and only if the corresponding dual family $\tau'$ in $\ds \calL(Y^*)$ 
	%\nin Next we use, for the first and only time, that we are taking $Y = L^q(\Omega),\ 1 < q < \infty$, so that we can invoke the complete duality for the $R$-boundedness on the space $L^q(\Omega),\ 1 < q < \infty,$ with $\ds Y^* = L^{q'}(\Omega), \ \rfrac{1}{q} + \rfrac{1}{q'} = 1$. We have \cite[Corollary 2.11, p 90]{KW:2004} that 
	\begin{equation}\label{3.10}
	\tau' = \left\{ tR(it,A^*_{_F}), \ t \in \BR\backslash \{0\} \right\} \text{ is } R\text{-bounded}.
	\end{equation}
	\nin This result follows from \cite[Proposition 8.4.1 p. 211]{HNVW:2016} which shows such equivalence in $K$-convex spaces, combined with \cite[Ex 7.4.8, p 113]{HNVW:2016} stating that a UMD space is $K$-convex. The special case of such duality with respect to the space $Y = L^q(\Omega), \ 1 < q < \infty$, with $Y^* = L^{q'}(\Omega), \ \rfrac{1}{q} + \rfrac{1}{q'} = 1$ is given in \cite[Corollary 2.11, p 90]{KW:2004}. But the $R$-boundedness property in \eqref{3.10} is equivalent by the same result \cite[Theorem 1.11, p 76]{KW:2004}, \cite[Theorem, p 198]{We:2001}, to the property that $\ds A^*_{_F} \in MReg(L^p(0,T;Y^*))$, and this is true by Proposition \ref{prop-3.1}. In conclusion, we have $\ds A_{_F} \in MReg(L^p(0,T;Y))$, and Theorem \ref{Thm-1.2}, part a) is proved.\\
	
	\nin \textit{Part b)} If it is known that the s.c. analytic semigroup $\ds e^{A_{_F}t}, \ t \geq 0$ on $Y$ is uniformly stable, then we can take $T = \infty$ by invoking \cite[Theorem 5.2, p 307]{Dore:2000}: $\ds A_{_F} \in MReg(L^p(0,\infty;Y))$.\\
	
	\nin \textit{Part c)} is now obvious as $\ds B \in \calL(Y)$. \qed
	
	\begin{clr}
		In a UMD space $Y$, maximal $L^p$-regularity of $A: Y \supset \calD(A) \longrightarrow Y$ is equivalent to maximal $L^p$-regularity for $A^*: Y^* \supset \calD(A^*) \longrightarrow Y^*$.
	\end{clr}
	\nin This is contained in the proof given in Step 4 above.
	
	\section{A classical parabolic equation with finite dimensional boundary feedback control: maximal $L^p$-regularity on $Y = L^q(\Omega), 1 < q < \infty$.}\label{Sec-4}
	
	\subsection{Open and closed-loop boundary control problem}
	Let $\Omega$ be an open bounded domain in $\BR^d, \ d \geq 2,$ with sufficiently smooth boundary $\Gamma = \partial \Omega$. Let $\omega$ be an arbitrary small open smooth subset of the interior $\Omega$, $\omega \subset \Omega$, of positive measure.	
	\begin{center}
		\begin{tabular}{c}
			\begin{tikzpicture}[x= 10pt,y = 10pt,>=stealth, scale=0.4]
			\draw[thick]
			(-15,5)
			.. controls (-20,-9) and (-5,-18) .. (15,-7)
			.. controls  (24,-1) and (15,13).. (6,6)
			.. controls (3,3) and (-3,3) .. (-6,6)
			.. controls (-9,9) and (-13,9) .. (-15,5);
			\draw (15,-1) node[scale = 1] {$\Omega$};
			\draw (15,-10) node[scale = 1] {$\Gamma$};
			\draw[fill, line width = 1pt, ,opacity = .25] (4.4,-1) circle (2);
			\draw (-1.2,-6.2) node[scale = 1] {$\omega$};
			\draw[->,line width = 1pt]  (-0,-5.2) -- (2,-3);
			\end{tikzpicture}\\
			\mbox{Fig 1: Internal subportion $\omega$.}
			\end{tabular}
	\end{center}
	\nin For notational simplicity and space constraints, we shall focus on the canonical case of the Laplacian translated, in order to make the original boundary homogeneous problem (\hyperref[4.1]{4.1a-b-c}) with $f \equiv 0$ unstable. This will then introduce the boundary feedback stabilization problem that will ultimately be an illustration of the abstract Theorem \ref{Thm-1.2} with $T = \infty$ in Section \ref{Sec-4.3}. Without uniform stabilization, the boundary feedback problem (\hyperref[4.3]{4.3a-b-c}) will claim maximal $L^p$-regularity only for $T < \infty$ in Theorem \ref{Thm-4.1}. The treatment extends to second order (say), uniformly strongly elliptic operators. Thus, we consider the following parabolic problem in the unknown $y(t,x), \ x \in \Omega$, initially with open loop boundary control $f$ in the Dirichlet B.C.
	\begin{subequations}\label{4.1}
	\begin{empheq}[left=\empheqlbrace]{align}
	y_t &= (\Delta + c^2)y & \text{ in } Q \equiv (0,T] \times \Omega \label{4.1a}\\
	y|_{t = 0} &= y_0 & \text{ in } \Omega \label{4.1b}\\
	\left. y \right|_{\Sigma} &= f & \text{ in } \Sigma \equiv (0,T] \times \Gamma \label{4.1c}
	\end{empheq}
	\end{subequations}
	\nin Our goal is to convert the open loop control system \eqref{4.1} into a closed loop feedback control system. We choose the open loop control $f$ to be expressed as a finite dimensional feedback operator $F$ of the form
	\begin{equation}\label{4.2}
		f = Fy = \sum_{k=1}^{K} \ip{y}{w_k}_{L^2(\omega)} g_k,
	\end{equation}
	with given vectors $w_k \in L^2(\omega), \ g_k \in L^q(\Gamma), \ 1 < q < \infty,$ so that corresponding closed loop feedback control problem is
	\begin{subequations}\label{4.3}
		\begin{empheq}[left=\empheqlbrace]{align}
		y_t &= (\Delta + c^2)y & \text{ in } Q \equiv (0,T] \times \Omega \label{4.3a}\\
		y|_{t = 0} &= y_0 & \text{ in } \Omega \label{4.3b}\\
		\left. y \right|_{\Sigma} &= \sum_{k=1}^{K} \ip{y}{w_k}_{L^2(\omega)} g_k, & \text{ in } \Sigma \equiv (0,T] \times \Gamma \label{4.3c}
		\end{empheq}
	\end{subequations}
	\nin Our basic function space is $Y \equiv L^q(\Omega), \ 1 < q < \infty$.
	
	\subsection{Abstract model of the closed loop system \eqref{4.3}. Verification of Theorem \ref{Thm-1.2}, $T < \infty$.}
	
	\nin We introduce the translated Dirichlet Laplacian and corresponding Dirichlet map.
	\begin{equation}\label{4.4}
		\calA_{_{tr}} \varphi = (\Delta + c^2)\varphi, \ \calA_{_{tr}}: Y \supset \calD(\calA_{_{tr}}) = \left\{ \varphi\in W^{2,q}(\Omega): \left. \varphi \right|_{\Gamma} = 0 \right\} \longrightarrow Y.
	\end{equation}
	\vspace{-.8cm}
	\begin{subequations}\label{4.5}
		\begin{eqnarray}
			\phi = D g \iff \left\{ (\Delta + c^2)\phi \equiv 0 \text{ in } \Omega, \ \left. \varphi \right|_{\Gamma} = g \right\} \label{4.5a}\\
			D: L^q(\Gamma) \longrightarrow W^{\rfrac{1}{q},q}(\Omega) \subset \calD \left( (-\calA )^{\rfrac{1}{2q}}  \right)  \label{4.5b}
		\end{eqnarray}
	\end{subequations}
	\nin where
	\begin{equation}\label{4.6}
	\calA \varphi = \Delta \varphi, \quad \calD(\calA_{_{tr}}) = \calD(\calA)
	\end{equation}	
	\nin is a suitable translation of $\ds \calA_{_{tr}}$, so that the fractional powers $\ds (-\calA)^{\theta}, \ 1 < \theta < \infty$, are defined by complex interpolation \cite{A:1975}. The operator $\calA_{_{tr}}$ in \eqref{4.4} has compact resolvent on $Y = L^q(\Omega)$ and is the generator of a s.c. analytic semigroup $\ds e^{\calA_{_{tr}} t}$ on $Y \equiv L^q(\Omega)$ \cite[Example, p101]{Fri:1976}. Returning to problem \eqref{4.1} and using the definition of the Dirichlet map $D$ in \eqref{4.5}, we can rewrite Eq \eqref{4.1a} as
	\begin{equation}\label{4.7}
		y_t = (\Delta + c^2)(y - D f) \text{ in } Q, \quad \left. [y - D f] \right|_{\Gamma} = 0.
	\end{equation}
	\nin Hence, the abstract version of the open-loop system \eqref{4.1} is
	\begin{equation}\label{4.8}
		y_t = \calA_{_{tr}}(y - D f) \text{ on } Y \equiv L^q(\Omega).
	\end{equation}
	\nin Next, returning to \eqref{4.2} with $\ds F \in \calL(L^2(\omega), L^q(\Gamma))$, we see that the abstract version \eqref{4.8} of the closed-loop system \eqref{4.3} specializes to
	\begin{equation}\label{4.9}
		y_t = \calA_{_{tr}}(I - D F) y = A_{_{F,tr}} y, \ y(0) = y_0, \text{ on } Y \equiv L^q(\Omega).
	\end{equation}
	\nin We next verify that the boundary feedback closed loop control problem (\hyperref[4.3a]{4.3a-c}) that is, its abstract model \eqref{4.9}, satisfies Theorem \ref{Thm-1.2} for $T < \infty$.
	\begin{thm}\label{Thm-4.1}
		Let $1 < q < \infty, \ w_k \in L^2(\omega), \ g_k \in L^q(\Gamma)$.
		\begin{enumerate}[(i)]
			\item The feedback operator in \eqref{4.9}
			\begin{subequations}\label{4.10}
				\begin{align}
					A_{_{F,tr}} &= \calA_{_{tr}}(I - D F)\\
					L^q(\Omega) &\supset \calD(A_{_{F,tr}}) = \left\{ x \in L^q(\Omega): (I - D F)x \in \calD(\calA_{tr}) \right\}
				\end{align} 
			\end{subequations}
		\nin is the generator of a s.c. analytic semigroup $\ds e^{A_{_{F,tr}} t}$ on $Y \equiv L^q(\Omega), \ t \geq 0$.
		\item Moreover, $A_{_{F,tr}}$ has maximal $L^p$-regularity on $Y \equiv L^q(\Omega)$ up to $T < \infty$,
		\begin{equation}
		A_{_{F,tr}} \in MReg \left( L^p(0,T; Y) \right).
		\end{equation}
		\end{enumerate}
	\end{thm}
	\begin{proof}
		\begin{enumerate}[(i)]
			\item Of course part (ii) implies part (i). But the direct proof of part (i) is more direct. See \cite{LT:1983.1}, \cite{LT:1983.2}, \cite{RT:1980} for $q = 2$. Appropriate modifications yield the desired conclusion a) also for $1 < q < \infty$.
			\item We need to verify assumptions \Hi \ through \Hv \ of Section \ref{Sec-1}, except for boundedness of $\ds e^{-A^*t}$ on $Y^*$ in \Hiii, so that the maximal $L^p$-regularity for the operator $\calA_{_{F,tr}}$ in \eqref{4.10} will hold for $T < \infty$. \Hi \ Since $Y \equiv L^q(\Omega), \ 1 < q < \infty$, assumption \Hi \ holds true. \Hii \ This is a-fortiori true, since $\ds \calA_{_{tr}}$ is the generator of a s.c., analytic semigroup $\ds e^{\calA_{_{tr}} t}$ on $Y \equiv L^q(\Omega), \ t \geq 0$. \Hiii \ $Y \equiv L^q(\Omega), \ 1 < q < \infty$, is reflexive and $\ds Y^* = (L^q(\Omega))^* = L^{q'}(\Omega),\ \rfrac{1}{q} + \rfrac{1}{q'} = 1$. Moreover the operator $\ds \calA_{_{tr}}^*$
			\begin{equation}\label{4.12}
				\calA_{_{tr}}^* \varphi = (\Delta + 2c^2) \varphi,\ L^{q'}(\Omega) \supset \calD \left(\calA_{_{tr}}^* \right) = \left\{ \varphi\in W^{2,q'}(\Omega): \left. \varphi \right|_{\Gamma} = 0 \right\}
			\end{equation}
			\nin is also the generator of a s.c., analytic semigroup $\ds e^{\calA_{_{tr}}^* t}$ on $Y^*, \ t \geq 0$. In addition, it is well-known that $\calA_{_{tr}}^*$ has maximal $L^p$-regularity on $Y^*$: $\ds \calA_{_{tr}}^* \in MReg \left( L^p(0, T; Y^*) \right)$. Thus \Hiii \ holds true (without boundedness). \Hiv \ We take $U = L^q(\Gamma)$. Then \eqref{4.5b} verifies \Hiv \ for $D$ with $\ds \gamma = \rfrac{1}{2q}$. \Hv \ We are actually taking $A_o = 0$ in the present illustration. Thus, \Hi-\Hv \ have been verified and Theorem \hyperref[Thm-1.2]{1.2 a)} yields our present part (ii) of Theorem \ref{Thm-4.1}: maximal $L^p$-regularity up to any $T < \infty$.
		\end{enumerate}
	\end{proof}

	\subsection{Case $T = \infty$. Uniform stabilization of problem (\hyperref[4.3]{4.3a-c}), by boundary feedback control $f = Fy$ as in \eqref{4.2}, for suitable $\ds w_k \in L^2(\omega), \ g_k \in L^q(\Gamma)$.}\label{Sec-4.3}
	
	\nin In the present subsection, under suitable assumptions, we seek to specialize the class of localized interior vectors $w_k \in L^2(\omega)$ and boundary vectors $g_k \in L^q(\Gamma)$, so that the s.c. analytic semigroup $\ds e^{A_{_{F,tr}}t}, \ t \geq 0$, on $Y$, guaranteed by Theorem \hyperref[Thm-4.1]{4.1 (i)}  is, in addition, uniformly stable on $Y$. This goal can be rephrased as a uniform stabilization problem for the open-loop parabolic system (\hyperref[4.1]{4.1a-c}), by virtue of a suitable feedback control $f = Fy$ in \eqref{4.2}, for suitable vectors $\ds \{ w_k, g_k \}_{k = 1}^K$. Moreover, we seek $K$ to be minimal number. 
	A solution of this problem for $q = 2$ and $\omega$ replaced by $\Omega$ was given in \cite[Theorem 2.1D and Theorem 2.4D]{RT:1980:2}.
	
	\begin{rmk}\label{rmk-4.1}
		The vectors $w_k$ are selected from the full rank conditions in \cite[(2.11)]{RT:1980:2} which hold true also with the $L^p(\Omega)$ inner-product in \cite[(2.11)]{RT:1980:2} replaced by the $L^p(\omega)$ inner-product \cite[Claim 3.3, p1458]{BT:2004}, due to the Unique Continuation Theorem in \cite[Lemma 3.7, p1466]{BT:2004}, \cite{RT2:2009}. Instead, the vectors $g_k$ are obtained by showing a moment problem such as \cite[(A.7)]{RT:1980:2}. Once uniform stability of $\ds e^{A_{_{F,tr}}t}$ on $Y$ is achieved, we can then conclude that the maximal $L^p$-regularity of $A_{_{F,tr}}$ can be pushed to $T = \infty$, hence improving (in this special setting) Theorem \ref{Thm-4.1}.
	\end{rmk}

	\begin{thm}\label{Thm-4.2}
		Under the setting of Remark \ref{rmk-4.1} regarding the special choice of the vectors $w_k \in L^2(\omega)$ and $g_k \in L^q(\Gamma)$, the s.c. analytic semigroup $\ds e^{A_{_{F,tr}}t}$ is uniformly stable on $Y$. Hence $\ds A_{_{F,tr}}$ has maximal $L^p$-regularity up to $\ds T = \infty: \ A_{_{F,tr}} \in MReg \left( L^p(0,\infty; Y) \right)$.
	\end{thm}

	\section{Linearization of Navier-Stokes equations with boundary feedback control: maximal $\bs{L}^p$-regularity on $\lso$ and $\Bso$ up to $T = \infty$.}\label{Sec-5}
	
	\subsection{Linearized controlled Navier-Stokes problem.}
	
	\nin \uline{Notation:} Bold notation refers to vector-valued ($d$-valued) quantities and corresponding spaces.\\  
	
	\nin This section is based on paper \cite{LPT.2} and its predecessors \cite{BLT1:2006}, \cite{BLT2:2007}, \cite{LT1:2015}, \cite{LT2:2015} which provides uniform stabilization near an unstable equilibrium solution $\bs{y}_e$ of the Navier-Stokes equations, $d = 2,3$ in closed-loop form, by virtue of a finite-dimensional feedback control pair $\{ \bs{v}, \bs{u} \}$ on $\{ \wti{\Gamma}, \omega \}$. Here see Fig 2, $\wti{\Gamma}$ is an arbitrary small connected portion of the boundary $\Gamma = \partial \Omega$, of a bounded sufficiently smooth domain $\Omega$ in $\BR^d,  \ d = 2,3,$ while $\omega$ is an arbitrarily small collar supported by $\wti{\Gamma}$.
	\begin{center}
		\begin{tabular}{c}
		\begin{tikzpicture}[x=10pt,y=10pt,>=stealth, scale=.53]
		\draw[thick]
		(-15,5)
		.. controls (-20,-9) and (-5,-18) .. (15,-7)
		.. controls  (24,-1) and (15,13).. (6,6)
		.. controls (3,3) and (-3,3) .. (-6,6)
		.. controls (-9,9) and (-13,9) .. (-15,5);
		\draw[thick,shade,opacity=.5]
		(-15.8,-2)
		.. controls (-11,-4) and (-10,-7) .. (-10,-10)
		.. controls (-15,-7) and (-15.5,-4) .. (-15.8,-2);
		\draw[thick] (-15.8,-2) .. controls (-15.5,-4) and (-15,-7) .. (-10,-10);%outer most line
		\draw[thick] (-14.3,-2.7) .. controls (-14.5,-4) and (-13,-7) .. (-10,-9);%middle line
		\draw[thick] (-13.2,-3.4) .. controls (-13.5,-4) and (-12,-7) .. (-10.2,-8);%innemost line		
		
		\draw[->,thin]  (-13.2,-5.9) -- (-18.9,4);
		
		\draw[->,line width = 1pt]  (-14.5,-5.9) -- (-20,4);
		\draw (-21,5) node {$v$};
		
		\draw[thin]  (-13.2,-5.9) -- (-6.5,-1.8);
		\draw(-17,5) node {$\tau(\xi)$};
		\draw(-6.5, -0.5) node {$\xi$};
		
		%labels
		\draw (-19.25,-7.75) node[scale=1] {$\omega$};
		\draw (-15,-8.5) node[scale=1] {$\wti{\Gamma}$};
		\draw (-15.8,-2) node {$\bullet$};
		\draw (-10,-10) node {$\bullet$};
		\draw[<-]  (-14,-5) -- (-19,-7);
		\end{tikzpicture}\\
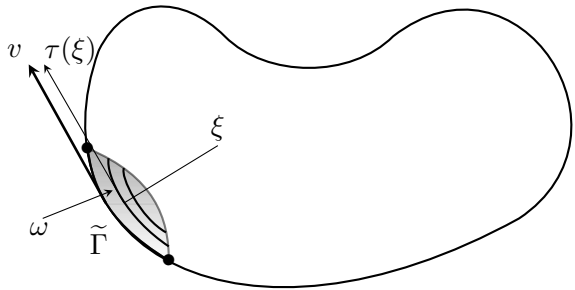
		\mbox{Fig 2: Internal localized collar $\omega$ of subportion $\wti{\Gamma}$ of boundary $\Gamma$}
	\end{tabular}
	\end{center}
	\nin The (eventually feedback) boundary control $\bs{v}$ acts tangentially over $\wti{\Gamma}$, while the (eventually feedback) interior control $\bs{u}$ acts ``tangential-like", that is, it acts in the tangential direction $\tau$, parallel to the boundary in the small boundary layer $\omega$. See Fig 2. To this end, a critical intermediary step towards the uniform stabilization of the nonlinear N-S system consists in considering the following linearized problem near the equilibrium solution $\bs{y}_e$, defined in Theorem \ref{Thm-5.1} below:
	\begin{subequations}\label{5.1}
		\begin{empheq}[left=\empheqlbrace]{align}
		\bs{w}_t - \nu_o \Delta \bs{w} + L_e(\bs{w}) + \nabla \chi - (m(x)\bs{u})\tau &= 0  &\text{ in } Q \label{5.1a}\\ 
		\text{div} \ \bs{w} &= 0   &\text{ in } Q \label{5.1b}\\
		\bs{w} &= \bs{v} &\text{ on } \Sigma \label{5.1c}\\
		\bs{w}(0,x) & = \bs{w}_0 (x) &\text{ on } \Omega \label{5.1d}
		\end{empheq}
	\end{subequations}
	\nin Here $m$ is the characteristic function of $\omega: \ m \equiv 1$ on $\omega$, $m \equiv 0$ on $\Omega \backslash \omega$, while $\nu_o > 0$ is the viscosity coefficient. $L_e$ is the first order Oseen perturbation
	\begin{equation}
		L_e(\bs{w}) = (\bs{y}_e \cdot \nabla)\bs{w} + (\bs{w} \cdot \nabla)\bs{y}_e \label{5.2}
	\end{equation} 
	\nin where $\bs{y}_e$ is the equilibrium solution, obtained from the following known result, the basic starting point of the analysis, see \cite[Theorem 5.iii, p58]{AR:2010} for $1 < q < \infty$ and \cite[Theorem 7.3, p59]{CF:1980} for $q = 2$.
	\begin{thm}\label{Thm-5.1}
		Consider the following steady-state Navier-Stokes equations in $\Omega$
		\begin{subequations}\label{5.3}
			\begin{align}
			-\nu_o \Delta \bs{y}_e +  (\bs{y}_e.\nabla)\bs{y}_e + \nabla \pi_e &= \bs{f} &\text{ in }  \Omega \label{5.3a}\\ 
			div \ \bs{y}_e &= 0 &\text{ in }  \Omega\label{5.3b} \\
			\bs{y}_e &= 0 &\text{ on } \Gamma. \label{5.3c}
			\end{align}
		\end{subequations}
		Let $1 < q < \infty$. For any $\bs{f} \in \bs{L}^q(\Omega)$ there exits a solution (not necessarily unique) $(\bs{y}_e,\pi_e) \in (\bs{W}^{2,q}(\Omega) \cap \bs{W}^{1,q}_{0}(\Omega)) \times (W^{1,q}(\Omega) \slash \BR)$.
	\end{thm}
	\nin \uline{\textbf{Case 1}: The equilibrium solution is unstable.} Instability of the equilibrium solution means that the corresponding Oseen operator $\calA_q$ in \eqref{5.11} below - which depends on $\bs{y}_e$ - has $N$ unstable eigenvalues: $\ds \ldots \leq Re~\lambda_{N+1} < 0 \leq Re~\lambda_N \leq \ldots \leq Re~\lambda_1$. To counteract such instability, \cite{LPT.2} seeks a boundary tangential control $\bs{v}$ acting with support on $\wti{\Gamma}$, and an interior control $\bs{u}$ acting tangential-like on $\omega$, of the preliminary form (for $\calF$ see \cite[Eqt (5.4)]{LPT.1})
	\begin{align}
	\bs{v} &= \sum^K_{k=1} \nu_k(t) \bs{f}_k, \quad \bs{f}_k \in \calF \subset \bs{W}^{2-\rfrac{1}{q},q}(\Gamma), \ q \geq 2,  \mbox{ so that } \bs{f}_k \cdot \nu = 0, \ \text{hence } \bs{v} \cdot \nu = 0 \mbox{ on } \Gamma \label{5.4}\\
	\bs{u} &= \sum_{k=1}^K \mu_k(t) \bs{u}_k, \quad \bs{u}_k \in \bs{W}^u_N \subset \lso, \quad \nu_k(t) = \text{scalar}, \ \mu_k(t) = \text{scalar,}\label{5.5}
	\end{align} 
	\nin -in fact, eventually in feedback from as in \eqref{5.17}, \eqref{5.18}. This will lead to the following boundary feedback closed loop PDE-system:
	
	\begin{subequations}\label{5.6}
		\begin{empheq}[left=\empheqlbrace]{align}
		\bs{w}_t - \nu_o \Delta \bs{w} + L_e(\bs{w}) + \nabla \chi &= m \left( \sum_{k = 1}^{K} \ip{P_N \bs{w}}{\bs{q}_k}_{\bs{W}^u_N} \bs{u}_k \right) \tau  &\text{ in } Q \label{5.6a}\\ 
		\text{div} \ \bs{w} &= 0   &\text{ in } Q \label{5.6b}\\
		\bs{w} &= \sum_{k = 1}^{K} \ip{P_N \bs{w}}{\bs{p}_k}_{\bs{W}^u_N} \bs{f}_k &\text{ on } \Sigma \label{5.6c}\\
		\bs{w}(0,x) & = \bs{w}_0 (x) &\text{ on } \Omega \label{5.6d}
		\end{empheq}
	\end{subequations}
	\nin to be further explained below. Qualitatively, the main result of the present Section \ref{Sec-5} is: \uline{for a suitable explicit selection of the boundary tangential vector $\bs{f}_k$ and interior vectors $\bs{q}_k, \bs{u}_k, \bs{p}_k \in \bs{W}^u_N$ as in \eqref{5.4}, \eqref{5.5} the resulting boundary feedback closed loop system (\hyperref[5.6]{5.6a-b-c-d}) generates a s.c. semigroup, which is analytic, uniformly stable, with generator that has maximal $L^p$-regularity up to $T = \infty$ in a suitable $\bs{L}^q$/Besov setting, $q > d$. to be identified blow}. Moreover, $ K = \max \{ $ geometric multiplicity of $\lambda_i, \ i = 1,\dots, N \}$. For the corresponding formal statements, we refer to Theorems \ref{Thm-5.2}-\ref{Thm-5.4} below. Maximal $L^p$-regularity will be an application of the abstract Theorem \ref{Thm-1.2} as it will be established in the present section. We note that in order to obtain uniform stabilization, and hence maximal $L^p$-regularity up to $T = \infty$, the interior tangential-like feedback control $\bs{u}$ in \eqref{5.5} ultimately acting on $\omega$, \uline{cannot be dispensed with}. This is due to a counter-example \cite{FL:1996} as explained in \cite{LPT.2}. The presence of such $\bs{u}$ is, abstractly, accounted for by the operator $B \in \calL(Y)$ in Theorem \ref{Thm-1.2}, part c). Uniform stabilization of problem (\hyperref[5.6]{5.6a-b-c-d}) rests critically at the outset of the (finite dimensional) analysis on a suitable Unique Continuation Property for a suitably overdetermined adjoint eigenproblem \cite{LT1:2015}, \cite{LT2:2015} to avoid the counterexample of \cite{FL:1996}. Here below, we shall put the PDE problem (\hyperref[5.6]{5.6a-b-c-d}) in the abstract setting of Theorem \ref{Thm-1.2}, part c). To this end, we need some preliminary background.
	
	\subsection{Preliminaries: Helmholtz decomposition}
	
	\begin{definition}\label{Def-5.1}
		Let $1 < q < \infty$ and $\Omega \subset \mathbb{R}^n$ be an open set. We say that the Helmholtz decomposition for $\bs{L}^q(\Omega)$ exists whenever $\bs{L}^q(\Omega)$ can be decomposed into the direct sum of the solenoidal vector space $\lso$ and the space $\bs{G}^q(\Omega)$ of gradient fields
		\begin{subequations}\label{5.7}
			\begin{equation}\label{5.7a}
				\bs{L}^q(\Omega) = \lso \oplus \bs{G}^q(\Omega),
			\end{equation}
			\begin{equation}\label{5.7b}
			\begin{aligned}
				\lso &= \overline{\{\bs{y} \in C_c^{\infty}(\Omega): \div \ \bs{y} = 0 \text{ in } \Omega \}}^{\norm{\cdot}_q}\\
				&= \{ \bs{g} \in \bs{L}^q(\Omega): \div \ \bs{g} = 0; \  \bs{g} \cdot \nu = 0 \text{ on } \partial \Omega \},\\
				& \hspace{3cm} \text{ for any locally Lipschitz domain } \Omega \subset \mathbb{R}^d, d \geq 2 \ \cite[p \ 119]{Ga:2011}\\
				\bs{G}^q(\Omega) &= \{ \bs{y} \in \bs{L}^q(\Omega): \bs{y} = \nabla p, \ p \in W_{loc}^{1,q}(\Omega) \ \text{where } 1 \leq q < \infty \}.
			\end{aligned}
			\end{equation}		
		\end{subequations}		
	\end{definition}
	\nin Both of these are closed subspaces of $\bs{L}^q(\Omega)$. The unique linear, bounded and idempotent (i.e. $P_q^2 = P_q$) projection operator $P_q: \bs{L}^q(\Omega) \longrightarrow \lso$ having $\lso$ as its range and $\bs{G}^q(\Omega)$ as its null space is called the Helmholtz projection. Under the present assumption of smoothness of $\Omega$ ($C^1$-smoothness is enough \cite{Ga:2011}), the Helmholtz projection is known to exist: The Helmholtz decomposition exists for $\bs{L}^q(\Omega)$ if and only if it exists for $\bs{L}^{q'}(\Omega)$, and we have: (adjoint of $P_q$) = $P^*_q = P_{q'}$ (in particular $P_2$ is orthogonal), where $P_q$ is viewed as a bounded operator $\ds \bs{L}^q(\Omega) \longrightarrow \bs{L}^q(\Omega)$, and $\ds P^*_q = P_{q'}$ as a bounded operator $\ds \bs{L}^{q'}(\Omega) \longrightarrow \bs{L}^{q'}(\Omega), \ \rfrac{1}{q} + \rfrac{1}{q'} = 1$.
	
	\subsection{Preliminaries: The Stokes and Oseen operators}
	
	\nin First, for $1 < q < \infty$ fixed, the Stokes operator $A_q$ in $\lso$ with Dirichlet boundary conditions  is defined by
	\begin{subequations}
	\begin{equation}\label{5.8a}
	A_q \bs{z} = -P_q \Delta \bs{z}, \quad
	\mathcal{D}(A_q) = \bs{W}^{2,q}(\Omega) \cap \bs{W}^{1,q}_0(\Omega) \cap \lso.
	\end{equation}
	The operator $A_q$ has a compact inverse $A_q^{-1}$ on $\lso$, hence $A_q$ has a compact resolvent on $\lso$. Moreover, it is well-known that $-A_q$ generates a s.c. analytic Stokes semigroup $e^{-A_qt}$ which is uniformly stable on $\lso$: there exist constants $M \geq 1, \delta > 0$ (possibly depending on $q$) such that 
	\begin{equation}\label{5.8b}
	\norm{e^{-A_qt}}_{\calL(\lso)} \leq M e^{-\delta t}, \ t > 0.
	\end{equation}
	\end{subequations}
	
	\nin It is equally well-known \cite{VAS:2001} that $-A_q$ has maximal $L^p$-regularity on $\lso$ up to $T = \infty$: $\ds -A_q \in MReg (L^p(0,\infty; \lso))$. Next, we recall from \eqref{5.2} the first order Oseen perturbation $L_e$
	\begin{equation}\label{5.9}
	L_e(\bs{z}) = (\bs{y}_e \cdot \nabla) \bs{z} + (\bs{z} \cdot \nabla) \bs{y}_e,
	\end{equation}	
	\noindent and define the first order operator $A_{o,q}$,
	\begin{equation}\label{5.10}
	A_{o,q} \bs{z} = P_q L_e(\bs{z}) = P_q[(\bs{y}_e \cdot \nabla )\bs{z} + (\bs{z} \cdot \nabla )\bs{y}_e], \ \mathcal{D}(A_{o,q}) = \mathcal{D}(A_q^{\rfrac{1}{2}}) = \bs{W}^{1,q}_0 (\Omega) \cap \lso,
	\end{equation}
	\nin Thus, $A_{o,q}A_q^{-\rfrac{1}{2}}$ is a bounded operator on $\lso$, and thus $A_{o,q}$ is bounded on $\ds \calD \big(A_q^{\rfrac{1}{2}} \big)$. This leads to the definition of the Oseen operator
	\begin{equation}\label{5.11}
	\calA_q  = - (\nu_o A_q + A_{o,q}), \quad \calD(\calA_q) = \calD(A_q) \subset \lso
	\end{equation}
	also with compact resolvent. Moreover $\calA_q$ generates a s.c. analytic semigroup $\ds e^{\calA_q t}$ on $\lso, \ t \geq 0$.
		
	\subsection{Preliminaries: Well-posedness in the $L^q$-setting of the non-homogeneous stationary Oseen problem: the Dirichlet map $D:$ boundary $\longrightarrow$ interior.}
	
	\noindent We follow \cite{BLT1:2006}, \cite{LT1:2015}, \cite{LT2:2015}, \cite{LPT.2}. Recalling the first order operator $L_e(\bs{\psi}) = (\bs{\psi} \cdot \nabla) \bs{y}_e + (\bs{y}_e \cdot \nabla) \bs{\psi} $ from \eqref{5.9} and introducing the differential expression $\BA \bs{\psi} = -\nu_0 \Delta \bs{\psi} + L_e(\bs{\psi})$, we consider the stationary, boundary non-homogeneous Oseen problem on $\Omega$: 	
	\begin{subequations}\label{5.12}
		\begin{align}
		\BA \bs{\psi} + \nabla \pi^* &= -\nu_o \Delta \bs{\psi} + L_e(\bs{\psi}) + \nabla \pi^* = 0 \label{5.12a}\\
		\begin{picture}(0,0)
		\put(-35,10){ $\left\{\rule{0pt}{20pt}\right.$}\end{picture}
		\text{ div } \bs{\psi} &= 0 \text{ in } \Omega; \quad \bs{\psi} = \bs{g} \text{ on } \Gamma, \ \bs{g} \cdot \nu = 0 \text{ on } \Gamma. \label{5.12b}
		\end{align}
	\end{subequations}
	\nin Problem \eqref{5.12} may not define a unique solution $\bs{\psi}$; that is, the operator $\bs{g} \to \bs{\psi}$ may have a nontrivial (finite dimensional) null space. To overcome this, one replaces in \eqref{5.12} the differential expression $\BA \bs{\psi} = -\nu_o \Delta \bs{\psi} + L_e (\bs{\psi})$  with its translation $ k +\BA$, for a positive constant $k$, sufficiently large as to obtain a unique solution $\bs{\psi}$. In line with the considerations made in \cite{LPT.2} and also in the name of simplicity of notation, we are here justified to admit henceforth that problem \eqref{5.12} (with $k = 0$) defines a unique solution $\psi$. We shall then denote the ``Dirichlet" map $ \bs{g} \longrightarrow \bs{\psi}$ by $D: \ D \bs{g} = \bs{\psi}$ in the notation of \eqref{5.12}. More precisely, define
	\begin{equation}\label{5.13}
		\bs{U}_q = \big\{ \bs{v} \in \bs{L}^q(\Gamma): \bs{v} \cdot \nu = 0 \text{ on } \Gamma \big\}.
	\end{equation}
	\nin Then with reference to problem \eqref{5.12} we have, recalling \cite[(0.2.17), p XXI]{W:1985}
	%\begin{equation}\label{5.14}
	%	\bs{\psi} = D \bs{v}
	%\end{equation}
	\begin{equation}\label{5.14}
		\bs{\psi} = D \bs{v}, \ \bs{v} \in \bs{U}_q \longrightarrow D \bs{g} \in \bs{W}^{\rfrac{1}{q},q}(\Omega) \cap \lso \subset \calD \Big( A_q^{\rfrac{1}{2q} - \varepsilon} \Big)
	\end{equation}
	\begin{equation}\label{5.15}
		\text{or} \quad A_q^{\rfrac{1}{2q} - \varepsilon} D \in \calL(\bs{U}_q, \lso).
	\end{equation}
	
	\subsection{Abstract model of the linearized $\bs{w}$-problem \eqref{5.1}.}
	
	\nin After the above background, we can finally give the abstract model (in additive form) of the linearized $\bs{w}$-problem in \eqref{5.1} in PDE-form still for $1 < q < \infty$. It is given by \cite{LT1:2015}, \cite{LT2:2015}, \cite{LPT.2}
	\begin{equation}\label{5.16}
	\begin{aligned}
	\bs{w}_t - \calA_q \bs{w} + \calA_{ext, q} D \bs{v} - P_q \big[ (m \bs{u}) \tau \big] = 0 &\text{ on } \big[ \calD(\calA_q^*) \big]'\\
	\begin{picture}(0,0)
	\put(-60,8){ $\left\{\rule{0pt}{20pt} \right.$}\end{picture}
	\bs{w}(x,0) = \bs{w}_0(x) = \bs{y}_0(x) - \bs{y}_e &\text{ in } \lso.
	\end{aligned}
	\end{equation}
	\nin In this section, $\calA_{ext, q}$ is the extension of $\calA_q$ in \eqref{5.11} from $\lso \to \big[ \calD(\calA_q^*) \big]'$.\\
	
	\nin \textbf{The operator $\BA_{_{F,q}}$ defining the linearized $\bs{w}$-problem in feedback form.}\\
	
	\nin Paper \cite{LPT.2} constructs \uline{suitable controllers} $\{ \bs{v}, \bs{u} \}$, this time in feedback form and thus going beyond \eqref{5.4}, \eqref{5.5}, with tangential boundary controller $\bs{v}$ supported on $\wti{\Gamma}$, and the tangential-like interior controller $\bs{u}$ supported on $\omega$ of the form		
	\begin{align}
	\bs{v} = F \bs{w} &= \sum_{k=1}^K \big<P_N \bs{w}, \bs{p}_k \big>_{_{\bs{W}^u_N}} \bs{f}_k, \quad \bs{f}_k \in \calF \subset \bs{W}^{2-\rfrac{1}{q},q}(\Gamma), \ \bs{p}_k \in (\bs{W}^u_N)^* \subset \lo{q'}, \ q \geq 2 \nonumber \\ & \hspace{5cm} \ \bs{f}_k \cdot \nu |_{\Gamma} = 0;  \text{ hence } \bs{v} \cdot \nu|_{\Gamma} = 0, \  \bs{f}_k \text{ supported on } \wti{\Gamma} \label{5.17}\\
	\bs{u} = J \bs{w} &= \sum_{k = 1}^K \big< P_N \bs{w}, \bs{q}_k \big>_{_{\bs{W}^u_N}} \bs{u}_k, \quad \bs{q}_k \in (\bs{W}^u_N)^* \subset \lo{q'},\ \bs{u}_k \text{ supported on } \omega. \label{5.18}
	\end{align}
	 Once inserted, this time, in the linear abstract $\bs{w}$-problem \eqref{5.16}, such $\bs{v}$ and $\bs{u}$ in \eqref{5.17}, \eqref{5.18} yield the linearized feedback dynamics driven by the dynamical feedback stabilizing operator $\ds \BA_{_{F,q}}$ below		
	\begin{align}
	\frac{d \bs{w}}{dt} &= \calA_q \bs{w} - \calA_q D \Bigg( \sum_{k = 1}^{K} \big< P_N \bs{w}, \bs{p}_k \big>_{_{\bs{W}^u_N}} \bs{f}_k \Bigg) + P_q \Bigg ( m \Bigg( \sum_{k = 1}^{K} \big< P_N \bs{w}, \bs{q}_k \big>_{_{\bs{W}^u_N}} \bs{u}_k \Bigg) \tau \Bigg) \equiv \BA_{_{F,q}} \bs{w}, \label{5.19}\\
	\frac{d \bs{w}}{dt} &= \calA_q \bs{w} - \calA_q D F \bs{w} + P_q m(J \bs{w}) \equiv \BA_{_{F,q}} \bs{w}. \label{5.20}
	\end{align}	
	\nin Eq \eqref{5.19} is the abstract version of the boundary feedback problem (\hyperref[5.6]{5.6a-d}) in PDE-system.
	\nin More specifically $\ds \BA_{_{F,q}}$ is rewritten as
	\begin{equation}\label{5.21}
	\BA_{_{F,q}} = A_{_{F,q}} + B: \lso \supset \calD \big( \BA_{_{F,q}} \big) \longrightarrow \lso, \ q \geq 2 
	\end{equation}
	%	\vspace{-0.8cm}
	\begin{subequations}\label{5.22}
		\begin{align}
		A_{_{F,q}} &= \calA_q (I - DF) \ : \ \lso \supset \calD(A_{_{F,q}}) \longrightarrow \lso, \ q \geq 2 \label{5.22a}\\
		\begin{picture}(0,0)
		\put(-15,9){$\left\{\rule{0pt}{21pt}\right.$}\end{picture}
		\calD(A_{_{F,q}}) &= \big\{ \bs{h} \in \lso: \bs{h} - DF \bs{h} \in \calD(\calA_q) = \bs{W}^{2,q}(\Omega) \cap \bs{W}^{1,q}_0(\Omega) \cap \lso \big\} = \calD(\BA_{_{F,q}}) \label{5.22b}
		\end{align}
	\end{subequations}
		\vspace{-0.8cm}
	\begin{subequations}\label{5.23}
		\begin{multline}
		F(\cdot)  = \sum_{k=1}^K \big< P_N \ \cdot, \bs{p}_k \big>_{_{\bs{W}^u_N}} \bs{f}_k \in \bs{W}^{2-\rfrac{1}{q},q}(\wti{\Gamma}); \\ B(\cdot)  = P_q \bigg( m \bigg( \sum_{k=1}^K \big< P_N \ \cdot, \bs{q}_k \big>_{_{\bs{W}^u_N}} \bs{u}_k \bigg) \tau \bigg) \in \lso \label{5.23a}
		\end{multline}	
		\begin{equation}\label{5.23b}
		F \in \calL(\lso, L^q(\wti{\Gamma})); \quad B \in \calL(\lso), \ q \geq 2 .
		\end{equation}				
	\end{subequations}
	
	\subsection{Application of abstract results of Section \ref{Sec-1} to the linearized Navier-Stokes boundary feedback problem \eqref{5.1} in the abstract form \eqref{5.20}.}
	
	\nin The operator $\BA_{_{F,q}}$ on $\lso, q \geq 2,$ in \eqref{5.22} is of the same form as the abstract operator $A_{_F}$ in \eqref{1.6a} under the following correspondence:
	
	\begin{enumerate}[(1)]
		\item The space $Y$ in \Hi  \ is now $\lso, \ q \geq 2,$ which is a UMD-space. Assumption \Hi \ holds true.
		\item The abstract operator $-A$ in \Hii is now the Stokes operator $-A_q$ in \eqref{5.8a}. As noted, $-A_q$ is the generator of a s.c. analytic semigroup $\ds e^{-A_q t}$ on $Y = \lso$, which moreover is uniformly stable by \eqref{5.8b}. It is equally classical that $-A_q$ has maximal $L^p$-regularity on $\lso$ up to $T = \infty$. So, a-fortiori, \Hii \ holds true \cite{VAS:1968}, \cite{VAS:1977}, \cite{VAS:1981}, \cite{VAS:1996}.
		\item The space $\lso, \ q \geq 2,$ is reflexive. The adjoint operator $-A_q^*$ in the $\lso \longrightarrow \lo{q'}$ duality pairing is given by 
		\begin{equation}\label{5.24}
			A_q^* \bs{f} = -P_{q'} \Delta \bs{f}, \ \calD(A_q^*) = \bs{W}^{2,q'}(\Omega) \cap \bs{W}^{1,q'}_0(\Omega)\cap \lo{q'},
		\end{equation}
		\nin and thus $-A_q^*$ generates a s.c. analytic uniformly stable semigroup $\ds e^{-A_q^* t}$ on $\ds Y^* = (\lso)' = \lo{q'}$. Moreover, such $-A_q^*$ has maximal $L^p$-regularity on $Y^*$ up to $T = \infty$. Thus assumption \Hiii \ holds true.
		\item The abstract Green map $G$ in \Hiv \ is now the Dirichlet map \eqref{5.14} and the abstract Banach space $U$ in \Hiv \ is now $\bs{U}_q$ as defined in \eqref{5.13}. The assumption \eqref{1.1} for $G$ is given by \eqref{5.14}, \eqref{5.15} with constant $\ds \gamma = \rfrac{1}{2q} - \varepsilon$. This way, assumption \Hiv \ holds true.
		\item The abstract operator $A_o$ in \Hv \ is the Oseen perturbation $A_{o,q}$ in \eqref{5.10}. Thus, \eqref{1.2a} holds true with $\ds 1 - \varepsilon = \rfrac{1}{2}$ by \eqref{5.10}. The abstract operator $\calA$ in \eqref{1.3} is the Oseen operator \eqref{5.11}. The operator $F$ in \eqref{1.4} is the operator in \eqref{5.17}.
	\end{enumerate}
	\nin We next recall that \cite{LPT.2} shows that one \uline{can construct explicitly}, vectors $\bs{p}_k, \bs{u}_k, \bs{f}_k, \bs{q}_k$ hence an operator $B$ in (\hyperref[5.23a]{5.23a-b}) such that the operator $\ds \BA_{_{F,q}} = A_{_{F,q}} + B$ generates a s.c. analytic semigroup on $\ds e^{\BA_{_{F,q}} t}$ on $\lso$, which moreover is uniformly stable,
	\begin{equation}\label{5.25}
		\norm{e^{\BA_{_{F,q}} t}}_{\calL(\lso)} \leq C_{\gamma_0} e^{-\gamma_0 t}, \quad t \geq 0, \ q \geq 2,
	\end{equation}
	\nin with decay rate $\ds \gamma_0 = \abs{Re \ \lambda_{N+1}} - \varepsilon, \ \lambda_{N+1}$ being the first unstable eigenvalue of $\calA_q$, see below \eqref{5.3c}. In order to achieve \eqref{5.25}, it is \textit{critical} to use a suitable operator $B$ as in \eqref{5.23a}, i.e. the interior tangential-like control $\bs{u}$, in view of the counter example \cite{FL:1996} to a required Unique Continuation Property even for the Stokes over-determined problem for $d = 2$. More insight is given in \cite{LPT.2}.\\
	
	\nin On the basis of the above considerations, in particular subject to the vectors $\bs{p}_k, \bs{u}_k, \bs{f}_k, \bs{q}_k$ as identified in \cite{LPT.4}, we can apply the abstract Theorem \ref{Thm-1.2} and obtain the next three results.
	
	\begin{thm}\label{Thm-5.2}
		\begin{enumerate}[a)]
			\item The operator $\ds \BA_{_{F,q}} = A_{_{F,q}} + B$ given by \eqref{5.21} has maximal $L^p$-regularity on $\lso$ up to $T = \infty$: $\ds \BA_{_{F,q}} \in MReg \big(\lplqs \big), \ q \geq 2$.
			\item The operator $A_{_{F,q}}$ in \eqref{5.22a} has maximal $L^p$-regularity on $\lso$ up to $T < \infty$:\\ $\ds A_{_{F,q}} \in MReg \big( L^p(0,T;\lso) \big), \ q \geq 2, \ T < \infty$.
		\end{enumerate}
	\end{thm}
	\nin A companion result, established in \cite[Theorem 11.4]{LPT.2} describes the action of semigroup $\ds e^{\BA_{_{F,q}}t}$ on the subspace % UNPUBLISHED THEOREM!!!
	\begin{subequations}\label{5.26}
		\begin{align}
		\Bto &=  \left\{ \bs{g} \in \Bso : \text{ div } \bs{g} = 0, \ \bs{g} \cdot \nu|_{\Gamma} = 0 \right\} \label{5.26a}\\
		&= \Bso \cap \lso, \quad 1 < p < \frac{2q}{2q-1} \label{5.26b}
		\end{align}
	\end{subequations}
	\nin of the Besov space
	\begin{equation}
		\Bso = \left( \bs{L}^q (\Omega), \bs{W}^{2,q}(\Omega) \right)_{1-\frac{1}{p},p} \hspace{3cm}
	\end{equation}
	\nin defined as a real interpolation space, as a specialization of the general formula
	\begin{equation}
		\bs{B}^s_{q,p}(\Omega) = \left( \bs{L}^q (\Omega), \bs{W}^{m,q}(\Omega) \right)_{\frac{s}{m},p} \hspace{3cm}
	\end{equation}
	\nin for $m = 2, s = \rfrac{2}{p}$.\\
	\begin{thm}{\cite[Theorem 11.4]{LPT.2}}\label{Thm-5.3} Consider now the original s.c. analytic  feedback semigroup $e^{\BA_{_{F,q}}t}$ on $\lso$, which is uniformly stable here by \eqref{5.25}. Let $\ds 1 < p < \rfrac{2q}{2q-1}, \ q \geq 2$. Then,
		\begin{align}
		e^{\BA_{_{F,q}}t}&: \text{ continuous } \Bto = \lqafq = \lqaq \label{5.29}\\
		&\longrightarrow \xipqs = L^p \big(0, \infty ; \bs{W}^{2,q}(\Omega) \big) \cap W^{1,p} \big( 0, \infty; \lso \big). \label{5.30}
		\end{align}
	\end{thm}% UNPUBLISHED THEOREM!!!
	\nin \uline{\textbf{Case 2}:} The literature reports physical situations where the volumetric force $f$ in \eqref{5.3a}, is actually replaced by $\nabla g(x)$; that is, $f$ is a conservative vector field. In this case, a solution to the stationary problem \eqref{5.3} is: $ \bs{y}_e \equiv 0, \pi_e = g $. Taking $ \bs{y}_e \equiv 0 $ (hence $L_e(\cdot) = 0$ by \eqref{5.2}) one obtains $A_{o,q} = 0, \ \calA_q = -A_q$ and the linearized $w$-equation \eqref{5.16} specializes to 
	\begin{equation}\label{5.31}
		\bs{\eta}_t + \nu_o A_q (\bs{\eta} - D \bs{v}) = P_q (m\bs{u}) \quad \text{in } \lso.
	\end{equation}
	\nin In this case, as discussed in \cite{LPT.2}, we can \uline{enhance at will the uniform stability} of the corresponding problem by the use \uline{only} of the tangential feedback finite dimensional control $\bs{v}$, as acting on the \uline{entire} boundary $\Gamma$. Thus we can take $\bs{u} \equiv 0$ in this case. With boundary feedback operator $F$ as in \eqref{5.17} except as acting now on the whole boundary $\Gamma$, the resulting, feedback operator is $(\nu_o = 1)$
	\begin{align}
		A_{_{F,q}} &= -A_q(I - DF) \label{5.32}\\
		F \cdot &= \sum_{k=1}^K \big<P_N \ \cdot \ , \bs{p}_k\big>_{_{\bs{W}^u_N}} \bs{f}_k, \quad \bs{f}_k \in \calF \subset \bs{W}^{2-\rfrac{1}{q},q}(\Gamma), \label{5.33}
	\end{align}
	\nin The corresponding closed-loop feedback system in PDE-form is
	\begin{subequations}
		\begin{empheq}[left=\empheqlbrace]{align}
			\bs{\eta}_t - \nu_o \Delta \bs{\eta} + \nabla \pi &= 0 &\text{ in } Q \label{5.34a}\\
			\text{ div } \bs{\eta} &= 0 &\text{ in } Q \label{5.34b}\\
			\left. \bs{\eta} \right|_{\Gamma} = F \bs{\eta} &= \sum_{k=1}^K \big<P_N \bs{\eta}, \bs{p}_k\big>_{_{\bs{W}^u_N}} \bs{f}_k &\text{ in } \Sigma \label{5.34c}
		\end{empheq}
	\end{subequations}
	\nin On the basis of the above considerations we obtain
	\begin{thm}\label{Thm-5.4}
		Let $\bs{y}_e = 0$. One can select the vectors $\bs{p}_k, \bs{f}_k$ in \eqref{5.33} so that the feedback operator $A_{_{F,q}}$ in \eqref{5.32} is the generator of a s.c. analytic semigroup $\ds e^{A_{_F}t}$ on $\lso$, which moreover has an arbitrary preassigned decay rate 
		\begin{equation}\label{5.35}
			\norm{e^{\BA_{_F}t}}_{\calL(\lso)} \leq M_r e^{-r t}, \quad t \geq 0.
		\end{equation}
		$r > 0$, preassigned. Finally, $A_{_F}$ has maximal $L^p$-regularity on $\lso, \ q \geq 0,$ up to $T = \infty$: $\ds A_{_F} \in MReg \big(\lplqs \big)$.
	\end{thm}

	\section{Linearization of the Boussinesq system with finite dimensional boundary feedback control: maximal $L^p$-regularity on $\bls \times L^q(\Omega)$ up to $T = \infty$.}\label{Sec-6}	
	
	\nin This section is based on paper \cite{LPT.4} which provides uniform stabilization near an unstable equilibrium solution $\bs{y}_e$ of the nonlinear Boussinesq system $d = 2, 3$ in closed-loop form, by virtue of a pair of finite-dimensional feedback controls $\{ v, \bs{u} \}$ acting on $\{ \wti{\Gamma}, \omega \}$. Here, see Fig 2, except that $\bs{u}$ is \uline{not} tangential-like in the present section, $\wti{\Gamma}$ is an arbitrary small connected position of the boundary $\ds \Gamma = \partial \Omega$ of a bounded, sufficiently smooth domain $\Omega$ in $\BR^d,\ d = 2,3,$ while $\omega$ is an arbitrary small collar supported by $\wti{\Gamma}$. To this end, a critical intermediary step - of interest to the present paper - consists in studying the following linearized Boussinesq system in PDE form near $\bs{y}_e$ in the variable $\ds \bs{w} = \{ \bs{w}_f, w_h \} \in \bls \times L^q(\Omega) \equiv \bs{W}^q_{\sigma}(\Omega)$:	
	
%	\begin{center}
%		\begin{tikzpicture}[x=10pt,y=10pt,>=stealth, scale=.53]
%		\draw[thick]
%		(-15,5)
%		.. controls (-20,-9) and (-5,-18) .. (15,-7)
%		.. controls  (24,-1) and (15,13).. (6,6)
%		.. controls (3,3) and (-3,3) .. (-6,6)
%		.. controls (-9,9) and (-13,9) .. (-15,5);
%		\draw[thick,shade,opacity=.5]
%		(-15.8,-2)
%		.. controls (-11,-4) and (-10,-7) .. (-10,-10)
%		.. controls (-15,-7) and (-15.5,-4) .. (-15.8,-2);
%		\draw[thick] (-15.8,-2) .. controls (-15.5,-4) and (-15,-7) .. (-10,-10);%outer most line
%		\draw[thick] (-14.3,-2.7) .. controls (-14.5,-4) and (-13,-7) .. (-10,-9);%middle line
%		\draw[thick] (-13.2,-3.4) .. controls (-13.5,-4) and (-12,-7) .. (-10.2,-8);%innemost line

%		\draw[->,thin]  (-13.2,-5.9) -- (-18.9,4);
%		\draw[thin]  (-13.2,-5.9) -- (-6.5,-1.8);
%		\draw(-18,5) node {$\tau(\xi)$};
%		\draw(-6.5, -0.5) node {$\xi$};
		
		%labels
%		\draw (-19.25,-7.75) node[scale=1] {$\omega$};
%		\draw (-14.5,-8) node[scale=1] {$\wti{\Gamma}$};
%		\draw (-15.8,-2) node {$\bullet$};
%		\draw (-10,-10) node {$\bullet$};
%		\draw[<-]  (-14,-5) -- (-19,-7);
%		\end{tikzpicture}
%	\end{center}

	\begin{subequations}\label{6.1}
		\begin{empheq}[left=\empheqlbrace]{align}
		\frac{d \bs{w}_f}{dt} - \nu \Delta \bs{w}_f + L_e(\bs{w}_f) - \gamma w_h \bs{e}_d + \nabla \chi &= m \bs{u}   &\text{ in } Q \label{6.1a}\\
		\frac{dw_h}{dt} - \kappa \Delta w_h + \bs{y}_e \cdot \nabla w_h + \bs{w}_f \cdot \nabla \theta_e &= 0 &\text{ in } Q \label{6.1b}\\
		\text{div } \bs{w}_f &= 0   &\text{ in } Q \label{6.1c}\\
		\bs{w}_f \equiv 0, \ w_h &\equiv v &\text{ on } \Sigma \label{6.1d}\\
		\bs{w}_f(0,\cdot) = \bs{w}_{f,0}; \quad w_h(0,\cdot) & = w_{h,0} &\text{ on } \Omega. \label{6.1e}
		\end{empheq}
	\end{subequations}	
	\nin with I.C. $\ds \{ \bs{w}_f(0), w_h(0) \} \in \bs{W}^q_{\sigma}(\Omega) \equiv \bls \times L^q(\Omega)$. Here, as in Section \ref{Sec-5}, $m$ is the characteristic function of $\omega: \ m \equiv 1$ on $\omega$, $m \equiv 0$ on $\Omega \backslash \omega$, while the first order Oseen perturbation $L_e$ is defined in \eqref{5.2}. The term $\bs{e}_d$ denotes the vector $(0,\dots,0,1)$, while $\kappa, \nu$ are physical constants. The original nonlinear Boussinesq system  models heat transfer in a viscous incompressible heat conducting fluid. It consists of  the Navier-Stokes equations (in the velocity vector) coupled with the convection-diffusion equation (for the scalar temperature). The equilibrium solution $\bs{y}_e$ is obtained from the following result, the basic starting point of our analysis. 
	\begin{thm} \label{Thm-6.1}
		Consider the following steady-state Boussinesq system in $\Omega$
		\begin{subequations}\label{6.2}
			\begin{empheq}[left=\empheqlbrace]{align}
			- \nu \Delta \mathbf{y}_e + (\mathbf{y}_e \cdot \nabla) \mathbf{y}_e - \gamma (\theta_e - \bar{\theta}) \mathbf{e}_d+ \nabla \pi_e &= \mathbf{f}(x)   &\text{in } \Omega \label{6.2a}\\
			-\kappa \Delta \theta_e + \mathbf{y}_e \cdot \nabla \theta_e &= g(x) &\text{in } \Omega \label{6.2b}\\
			\text{div } \mathbf{y}_e &= 0   &\text{in } \Omega \label{6.2c}\\
			\mathbf{y}_e = 0, \ \theta_e &= 0 &\text{on } \partial \Omega. \label{6.2d}
			\end{empheq}
		\end{subequations}
		Let $1 < q < \infty$. For any $\mathbf{f},g \in \mathbf{L}^q(\Omega), L^q(\Omega)$, there exists a solution (not necessarily unique) $(\mathbf{y}_e,\theta_e, \pi_e) \in (\mathbf{W}^{2,q}(\Omega) \cap \mathbf{W}^{1,q}_{0}(\Omega)) \times (W^{2,q}(\Omega) \cap W^{1,q}_{0}(\Omega)) \times (W^{1,q}(\Omega)/\mathbb{R})$.
	\end{thm}
	\noindent See \cite{Ac}, \cite{AAC.1}, \cite{AAC.2} for $ q \ne 2$. In the Hilbert space setting, see \cite{CF:1980}, \cite{FT:1984}, \cite{VRR:2006}, \cite{Kim:2012}.\\
	
	\nin \textbf{Instability of the equilibrium solution.} Instability of the equilibrium solution means that the operator $\BA_q$ in \eqref{6.14} below has a finite number, say $N$ unstable eigenvalues $\ds \dots \leq Re \ \lambda_{N+1} < 0 \leq Re \ \lambda_N \leq \dots \leq Re \ \lambda_1$. To counteract such instability, \cite{LPT.4} seeks a boundary control $v$ acting with support $\wti{\Gamma}$, and an interior control $\bs{u}$ acting on $\omega$, of the following feedback form
	\begin{align}
	v &= \sum^K_{k=1} \ip{P_N \bs{w}}{\bs{p}_k} f_k, \quad f_k \in \calF \subset \bs{W}^{2-\rfrac{1}{q},q}(\Gamma), \ \bs{p}_k \in (\bs{W}^u_N)^* \subset \lo{q'} \times L^q(\Omega), \ q \geq 2, \nonumber \\ & \hspace{11cm} f_k \text{ supported on } \wti{\Gamma} \label{6.3}\\
	\bs{u} &= \sum_{k=1}^K \ip{P_N \bs{w}}{\bs{p}_k} \bs{u}_k, \quad \bs{u}_k \in \what{\bs{L}}^q_{\sigma}(\Omega), \ \bs{q}_k (\bs{W}^u_N)^* \subset \lo{q'} \times L^q(\Omega), \quad \bs{u}_k(t) \text{ supported on } \omega \label{6.4}
	\end{align}
	\begin{multline*}
	\widehat{ \mathbf{L}}^q_{\sigma} (\Omega) \equiv \text{ any (d-1)-dimensional the space obtained from } \mathbf{L}^q_{\sigma} (\Omega) \text{ after omitting one specific} \\ \text{ co-ordinate, except the } \text{d\textsuperscript{th} coordinate from the vectors of }\mathbf{L}^q_{\sigma} (\Omega).
	\end{multline*} 
	\nin which, once inserted in \eqref{6.1d} and \eqref{6.1a} respectively yield the following feedback closed loop PDE-system	
	\begin{subequations}\label{6.5}
		\begin{empheq}[left=\empheqlbrace]{align}
		\frac{d \bs{w}_f}{dt} - \nu \Delta \bs{w}_f + L_e(\bs{w}_f) - \gamma w_h \bs{e}_d + \nabla \chi &= m \left( \sum_{k=1}^K \ip{P_N \bs{w}}{\bs{p}_k} \bs{u}_k \right)   &\text{ in } Q \label{6.5a}\\
		\frac{dw_h}{dt} - \kappa \Delta w_h + \bs{y}_e \cdot \nabla w_h + \bs{w}_f \cdot \nabla \theta_e &= 0 &\text{ in } Q \label{6.5b}\\
		\text{div } \bs{w}_f &= 0   &\text{ in } Q \label{6.5c}\\
		\bs{w}_f \equiv 0, \ w_h &\equiv \sum^K_{k=1} \ip{P_N \bs{w}}{\bs{p}_k} f_k &\text{ on } \Sigma \label{6.5d}\\
		\bs{w}_f(0,\cdot) = \bs{w}_{f,0}; \quad w_h(0,\cdot) & = w_{h,0} &\text{ on } \Omega. \label{6.5e}
		\end{empheq}
	\end{subequations}	
	\nin to be further explained below. Qualitatively, the main result of the present Section \ref{Sec-6} is: \uline{for a suitable explicit selection of the boundary vectors $f_k$ and the interior vectors $\bs{p}_k, \bs{q}_k, \bs{u}_k$ in \eqref{6.3}, \eqref{6.4} the resulting boundary feedback closed loop system (\hyperref[6.5]{6.5a-b-c-d}) generates a s.c. semigroup which is analytic, uniformly stable, with generator that has maximal $L^p$-regularity up to $T = \infty$ in a suitable functional setting to be identified below.} Moreover, $ K = \max \left\{ \text{geometric multiplicity of } \lambda_i, \ i = 1,\dots, N \right\}$. The formal statements will be given in Theorems \ref{Thm-6.4} and \ref{Thm-6.5} below.\\
	
	\nin The Helmholtz decomposition of Section \ref{Sec-5}, and related machinery, with projection $P_q$ applies now in the study of the linearized N-S equation \eqref{6.1a}. In particular, the space $\bls$ is defined in \eqref{5.7} and is the state space of the velocity vector. Next we define the coupling linear terms as bounded operators on $L^q(\Omega), \ \bls$ respectively, $q > d$:	
	\begin{align}
	\text{[from the N-S equation]} \quad \calC_{\gamma} h &= -\gamma P_q (h \mathbf{e}_d), \quad \calC_{\gamma} \in \calL (L^q(\Omega),\bls), \label{6.6}   \\
	\text{[from the heat equation]} \quad \calC_{\theta_e} \mathbf{z} &=  \mathbf{z} \cdot \nabla \theta_e, \quad \calC_{\theta_e} \in \calL(\bls,\lqo); \label{6.7}
	\end{align}
	 \nin Thus applying the Helmholtz projector $P_q$ to the coupled linearized $N-S$ equation \eqref{6.1a} and recalling the operator $\calA_q$ from \eqref{5.11} as well as \eqref{6.6}, we rewrite \eqref{6.1a} abstractly as
	 \begin{equation}\label{6.8}
	 	\frac{d \bs{w}_f}{dt} - \calA_q \bs{w}_f + C_{\gamma}w_h = P_q(m \bs{u}).
	 \end{equation}
	 \nin Next, with the goal of writing the abstract model for the coupled heat equation \eqref{6.1b}, we introduce the following operators
	 \begin{enumerate}[(i)]
	 	\item the heat operator $B_q$ in $L^q(\Omega)$ with homogeneous Dirichlet boundary conditions
	 	\begin{equation}\label{6.9}
	 	B_q h = -\Delta h, \quad \mathcal{D}(B_q) = W^{2,q}(\Omega) \cap W^{1,q}_0(\Omega);
	 	\end{equation}
	 	\item the first order operator $B_{o,q}$, corresponding to $B_q$:
	 	\begin{equation}\label{6.10}
	 	B_{o,q} h = \mathbf{y}_e \cdot \nabla h, \quad \mathcal{D}(B_{o,q}) = \mathcal{D}(B_q^{\rfrac{1}{2}}) \subset \lqo;
	 	\end{equation}
	 	\item the following operator for the heat component
	 	\begin{equation}\label{6.11}
	 	\calB_q  = - (\kappa B_q + B_{o,q}), \quad \calD(\calB_q) = \calD(B_q) \subset \lqo.
	 	\end{equation}
	 \end{enumerate} 
	 \nin If we take $v = 0$ in \eqref{6.1d}, the abstract version of the corresponding equation \eqref{6.1b} is, recalling \eqref{6.3}
	 \begin{equation}\label{6.12}
	 	\frac{d w_h}{dt} - \calB_q w_h + \calC_{\theta_e} \bs{w}_f = 0, \ \text{ for } v \equiv 0.
	 \end{equation}
	 \nin Thus, by \eqref{6.8} and \eqref{6.12}, the abstract model of the uncontrolled PDE-system (\hyperref[6.1a]{6.1a-e}) (that is, with $v \equiv 0$ and $\bs{u} \equiv 0$) is given by 
	 \begin{equation}\label{6.13}
	 \frac{d}{dt} \bbm \bs{w}_f \\ w_h \ebm = \BA_q \bbm \bs{w}_f \\ w_h \ebm \text{ in } \mathbf{W}^q_{\sigma}(\Omega) \equiv \lso \times L^q(\Omega), \text{ with } v \equiv 0, \bs{u} \equiv 0
	 \end{equation} 
	 \nin where the free dynamics operator $\ds \BA_q$ is given by 
	 \begin{multline}\label{6.14}
	 \BA_q = \bbm \calA_q & -\calC_{\gamma} \\ -\calC_{\theta_e} & \calB_q \ebm : \bs{W}^q_{\sigma}(\Omega) = \lso \times \lqo \supset \calD(\BA_q) = \calD(\calA_q) \times \calD(\calB_q) \\ = (\bs{W}^{2,q}(\Omega) \cap \bs{W}^{1,q}_{0}(\Omega) \cap \lso) \times (W^{2,q}(\Omega) \cap W^{1,q}_{0}(\Omega)) \longrightarrow \bs{W}^q_{\sigma}(\Omega).
	 \end{multline}	
	 \nin Next, in preparation for the abstract version of the fully controlled dynamics (\hyperref[6.1a]{6.1a-e}), we introduce the Dirichlet map $D$ \cite[p181]{L-T.1} with reference to the Dirichlet boundary controlled thermal equation \eqref{6.1b}
	 \begin{subequations}\label{6.15}
	 	\begin{empheq}[left=\empheqlbrace]{align}
	 		\psi = Dv \iff \left\{ \Delta \psi = 0 \text{ in } \Omega, \ \psi|_{\Gamma} = v \text{ on } \Gamma \right\} \label{6.15a}\\
	 		D:L^q(\Gamma) \longrightarrow W^{\rfrac{1}{q}, q}(\Omega) \subset \calD(B_q^{\rfrac{1}{2q} - \varepsilon}) \text{ continuously} \label{6.15b}\\
	 		B_q^{\rfrac{1}{2q} - \varepsilon}D \in \calL(L^q(\Gamma), L^q(\Omega)), \label{6.15c}
	 	\end{empheq}
	 \end{subequations}
 	\nin counterpart of \eqref{5.15}. Accordingly, we rewrite Eq \eqref{6.1b} as
 	\begin{equation}\label{6.16}
 		\frac{d w_h}{dt} - \kappa \Delta (w_h - Dv) + \bs{y}_e \cdot \nabla w_h + \bs{w}_f \cdot \nabla \theta_e = 0 \text{ in } Q
 	\end{equation}
 	\nin where $\ds [w_h - Dv]_{\Gamma} = 0$ by \eqref{6.1d}, \eqref{6.15a}. Accordingly, invoking the operators $B_q, \calC_{\theta_e}$ from \eqref{6.9}, \eqref{6.7}, we can rewrite Eq \eqref{6.16} abstractly as
 	\begin{equation}\label{6.17}
 		\frac{d w_h}{dt} + \kappa B_q (w_h - Dv) + B_{o,q} w_h + \calC_{\theta_e} \bs{w}_f = 0.
 	\end{equation}
 	\nin Thus, setting $\ds \bs{w} = \{ \bs{w}_f, w_h \}$ and combining Eqts \eqref{6.8} with Eq \eqref{6.17}, we obtain the abstract model of the controlled PDE-linearized Boussinesq system (\hyperref[6.1a]{6.1a-e}):
 	\begin{equation}\label{6.18}
 		\frac{d \bs{w}}{dt} =
 		\frac{d}{dt}
 		\bbm \bs{w}_f \\ w_h \ebm = \bbm \calA_q & -\calC_{\gamma} \\ -\calC_{\theta_e} & \calB_q \ebm \bbm \bs{w}_f \\ w_h \ebm + \bbm P_q (m\mathbf{u}) \\ \kappa B_{ext,q}Dv \ebm.
 	\end{equation}
 	\nin where $B_{ext,q}$ extends $B_q$ in \eqref{6.9} from $\lqo \to [\calD(B_q^*)]'$.
 	%\nin The Oseen operator $L_e(\cdot)$ in \eqref{6.1a} is defined in \eqref{5.2} in terms of the equilibrium solution $\bs{y}_e$. This is obtained from the following result, the basic starting point of our analysis. 	
 	
 	\subsection{Properties of the operator $\BA_q$ in \eqref{6.14}.}
 	
 	The following result collects basic properties of the operator $\BA_q$. It is essentially a corollary of Theorems A.3 and A.4 in \cite[Appendix A]{LPT.4} for the Oseen operator $\calA_q$, as similar results hold for the operator $\calB_q$, while the operator $\calC_{\gamma}$ and $\calC_{\theta_e}$ in the definition \eqref{6.14} of $\BA_q$ are bounded operators, see \eqref{6.6}, \eqref{6.7}. 
 	
 	\begin{thm}\label{Thm-6.2}
 		With reference to the Operator $\BA_q$ in \eqref{6.14}, the following properties hold true:
 		\begin{enumerate}[(i)]
 			\item $\ds \BA_q$ is the generator of strongly continuous analytic semigroup on $\mathbf{W}^q_{\sigma}(\Omega)$ for $t > 0$;
 			\item $\BA_q$ possesses the maximal $L^p$-regularity property on  $\mathbf{W}^q_{\sigma}(\Omega)$ over a finite interval:
 			\begin{equation}
 			\BA_q \in MReg(L^p(0,T;\bs{W}^q_{\sigma}(\Omega))), \ 0 < T < \infty.
 			\end{equation}
 			\item $\ds \BA_q$ has compact resolvent on $\bs{W}^q_{\sigma}(\Omega)$.
 		\end{enumerate}
 	\end{thm}
 	
 	\nin Next, we impose that the pair $\{ v, \bs{u} \}$ of controls be given in feedback form as in \eqref{6.3}, \eqref{6.4} \cite{LPT.4} repeated here as
 	\begin{align}
 		v &= F \ \cdot \ = \sum_{k = 1}^K \ip{P_N \ \cdot \ }{\bs{p}_k}f_k, \ f_k \in \calF \subset W^{2 - \rfrac{1}{q},q}(\Gamma),  \label{6.20} \\& \hspace{3cm} \bs{p}_k \in \left[ \left( \bs{W}^q_{\sigma}(\Omega) \right)^u_N \right]^* \subset \bs{L}^{q'}_{\sigma}(\Omega) \times L^q(\Omega), \ q \geq 2, f_k \text{ supported on } \wti{\Gamma}. \nonumber\\ 
 		\bs{u} &= J \ \cdot \ = P_q \left(m\sum_{k = 1}^K \ip{P_N \ \cdot \ }{\bs{q}_k} \bs{u}_k\right), \ \bs{q}_k \in \left[ \left( \bs{W}^q_{\sigma}(\Omega) \right)^u_N \right]^* \subset \bs{L}^{q'}_{\sigma}(\Omega) \times L^q(\Omega) \ \bs{u}_k \text{ supported on } \omega, \label{6.21}
 	\end{align}
 	\nin so that $F$ and $J$ are both bounded operators
 	\begin{equation}\label{6.22}
 		F \in \calL \left( \bs{W}^q_{\sigma}(\Omega), \bs{L}^q(\Gamma) \right); \ J \in \calL \left( \bs{W}^q_{\sigma}(\Omega), \bls \right) 
 	\end{equation}
 	\nin In \eqref{6.20}, \eqref{6.21}, $\ip{ \cdot }{ \cdot }$ denotes the duality paring $\ds \left( \bs{W}^q_{\sigma}(\Omega) \right)^u_N \rightarrow \left[ \left( \bs{W}^q_{\sigma}(\Omega) \right)^u_N \right]^*$ and the vectors $\ds \bs{p}_k, \bs{q}_k \in \left[ \left( \bs{W}^q_{\sigma}(\Omega) \right)^u_N \right]^* $. Substituting \eqref{6.20}, \eqref{6.21} into \eqref{6.18} yields the linearized $\bs{w}$-problem in feedback form
 	\begin{equation}\label{6.23}
 		\frac{d \bs{w}}{dt} = \bbm \calA_q & -\calC_{\gamma} \\[1mm] -\calC_{\theta_e} & \calB_q \ebm \bs{w} + \bbm \ds P_q \left( m \sum_{k = 1}^K \ip{P_N \bs{w}}{\bs{q}_k} \bs{u}_k \right) \\[5mm] \ds \kappa B_{ext,q}D \left( \sum_{k = 1}^K \ip{P_N \bs{w}}{\bs{p}_k} f_k \right) \ebm
 	\end{equation}
 	\nin or
 	\begin{equation}\label{6.24}
 	\frac{d \bs{w}}{dt} = \bbm \calA_q & -\calC_{\gamma} \\[1mm] -\calC_{\theta_e} & \calB_q \ebm \bs{w} + \bbm J \bs{w} \\ \kappa B_{ext,q}DF \bs{w} \ebm \equiv \BA_{_{F,q}} \bs{w}
 	\end{equation}
 %	\nin The operator $\ds \BA_{_{F,q}}$ defines the linearized $\bs{w}$-problem in feedback form, while $\ds \BA_q$ in ( \ ) is the free dynamics operator.
 	\nin Eq \eqref{6.23} is the abstract version of the boundary feedback problem (\hyperref[6.5]{6.5a-d}) in PDE form. Recalling \eqref{5.11} for $\calA_q$ and \eqref{6.11} for $\calB_q$, rewrite \eqref{6.24} with $\nu = \kappa = 1, \bs{w} = \left[ \bs{w}_1, w_2 \right] \in \bs{W}^q_{\sigma}(\Omega)$,
	 \begin{equation}\label{6.25}
	 	\frac{d \bs{w}}{dt} = \bbm -A_q \bs{w}_1 \\[1mm] -B_q \left( \bbm 0 \\ I_2 \ebm + DF \right) \bs{w} \ebm + \bbm -A_{o,q} & -\calC_{\gamma} \\[1mm] -\calC_{\theta_e} & -B_{o,q} \ebm \bs{w} + \bbm J \bs{w} \\ 0 \ebm \equiv \BA_{_{F,q}} \bs{w}
	 \end{equation}
	 \begin{equation}\label{6.26}
	 	\frac{d \bs{w}}{dt} = \BA_{_{F,q}} \bs{w} = \hat{\BA}_{_{F,q}} \bs{w} + \Pi \bs{w}
	 \end{equation}
	 \begin{equation}\label{6.27}
	 	\hat{\BA}_{_{F,q}} \bs{w} = \bbm -A_q \bs{w}_1 \\[1mm] -B_q \left( \bbm 0 \\ I_2 \ebm + DF \right) \bs{w} \ebm, \quad \Pi \bs{w} = \bbm -A_{o,q} & -\calC_{\gamma} \\[1mm] -\calC_{\theta_e} & -B_{o,q} \ebm \bs{w} + \bbm J \bs{w} \\ 0 \ebm.
	 \end{equation}
	 \begin{equation}\label{6.28}
	 	\calD \left( \hat{\BA}_{_{F,q}} \right) = \left\{ \bs{w} = \bbm \bs{w}_1 \\ w_2 \ebm \in \bs{W}^q_{\sigma}(\Omega) = \bls \times L^q(\Omega): \ \bs{w}_1 \in \calD(A_q),\ \left( \bbm 0 \\ I_2 \ebm + DF \right)\bs{w} \in \calD(B_q)  \right\}
	 \end{equation}
	 \begin{equation}\label{6.29}
	 	\calD(\Pi) = \calD(A_{o,q}) \times \calD(B_{o,q}).
	 \end{equation}
	 \subsection{Maximal $L^p$-regularity on $\bs{W}^q_{\sigma}(\Omega)$ of the linearized feedback operator $\ds \BA_{_{F,q}}$ up to $T = \infty$.}
	 
	 \nin With reference to the operator $\BA_{_{F,q}}$ in \eqref{6.24} or \eqref{6.25}, consider the following abstract dynamics
	 \begin{equation}
	 	\bs{\chi}_t = \BA_{_{F,q}} \bs{\chi} + q, \quad \bs{\chi}(0) = 0 \text{ in } \bs{W}^q_{\sigma}(\Omega)
	 \end{equation}
	 \begin{equation}
	 	\bs{\chi}(t) = \int_{0}^{t} e^{\BA_{_{F,q}}(t-s)}q(s)ds
	 \end{equation}
	 \nin The main theorem of the present Section \ref{Sec-6} is
	 \begin{thm}\label{Thm-6.3}
	 	With reference to the bounded operator $F$ and $J$ in \eqref{6.20}, \eqref{6.21}, let $T < \infty$. Then, the operator $\ds \BA_{_{F,q}}$ in \eqref{6.25} has maximal $L^p$-regularity on $\bs{W}^q_{\sigma}(\Omega)$ up to $T < \infty$; that is,
	 	\begin{equation}
	 		(L \bs{\chi})(t) = \int_{0}^{t} e^{\BA_{_{F,q}}(t-s)}\bs{\chi}(s)ds
	 	\end{equation}
	 \end{thm}
	 \nin continuous:
	 \begin{equation}
	 	L^p(0,T; \bs{W}^q_{\sigma}(\Omega)) \longrightarrow L^p \left(0,T; \calD \left(\BA_{_{F,q}} \right) \right)
	 \end{equation}
	 \nin so that continuously
	 \begin{equation}
	 	\bs{\chi} \in L^p \left(0,T; \calD \left(\BA_{_{F,q}} \right) \right) \cap W^{1,p}(0,T; \bs{W}^q_{\sigma}(\Omega)). 
	 \end{equation}
	 \subsection{The problem of feedback stabilization of the $\bs{w}$-dynamics \eqref{6.23}.}
	 
%	 \noindent \textbf{Basic assumption:} By Theorem \ref{Thm-6.2}, the operator $ \BA_q $ in \eqref{6.10} has the eigenvalues (spectrum) located in a triangular sector of well-known type. Then our basic assumption - which justifies the present section - is that such operator $\ds \BA_q$ is unstable: that is, $\BA_q$ has a finite number, say $N$, of eigenvalues $\lambda_1, \lambda_2 ,\lambda_3 ,\dots,\lambda_N$ on the complex half plane $\{ \lambda \in \mathbb{C} : Re~\lambda \geq 0 \}$ which we then order according to their real parts, so that	
%	 \begin{equation}\label{1.21}
%	 \ldots \leq Re~\lambda_{N+1} < 0 \leq Re~\lambda_N \leq \ldots \leq Re~\lambda_1,
%	 \end{equation}
	 
%	 \noindent each $\lambda_i, \ i=1,\dots,N$, being an unstable eigenvalue repeated according to its geometric multiplicity $\ell_i$. Let $M$ denote the number of distinct unstable eigenvalues $\lambda_i$ of $\BA_q$, so that $\ell_i$ is equal to the dimension of the eigenspace corresponding to $\lambda_i$. Instead, $\ds N = \sum_{i = 1}^{M} N_i$ is the sum of the corresponding algebraic multiplicity $N_i$ of $\lambda_i$, where $N_i$ is the dimension of the corresponding generalized eigenspace.\\
	\nin We return to the basic preliminary assumption of instability of the equilibrium solution, that is of the operator $\BA_q$ in \eqref{6.14}, see below \eqref{6.2}. The following result is proved in \cite[Theorem 2.1]{LPT.4}.
	 
	 \begin{thm}\label{Thm-6.4}
	 	With reference to the closed-loop feedback abstract dynamics $\bs{w}$ on \eqref{6.23}, whose PDE version is given by the system (\hyperref[6.5a]{6.5a-e}), we can select (in infinitely many ways) boundary vectors $f_k \in W^{2 - \rfrac{1}{q},q}(\wti{\Gamma})$ with support on $\wti{\Gamma}$, interior vectors $\bs{u}_k \in \what{\bs{L}}^q_{\sigma}(\omega)$ with support $\omega$ as well as vectors $\bs{p}_k, \bs{q}_k \in \left[ \left( \bs{W}^q_{\sigma}(\Omega) \right)^u_N \right]^* $ so that the s.c. analytic semigroup $\ds e^{\BA_{_{F,q}}t}$ is uniformly stable on $\bs{W}^q_{\sigma}(\Omega)$
	 	\begin{equation}\label{6.35}
	 		\norm{e^{\BA_{_{F,q}}t}}_{\calL \left(\bs{W}^q_{\sigma}(\Omega) \right)} \leq C e^{-\gamma_1 t}, \ t \geq 0
	 	\end{equation}
	 	\nin with constant $\gamma_1$, satisfying $\ds Re \ \lambda_{N+1} < \gamma_1 < 0$. Recall \eqref{6.4} for $\what{\bs{L}}^q_{\sigma}(\Omega)$; i.e. $\bs{u}_k$ is $(d-1)$-dimensional. 
	 \end{thm}
 	\begin{thm}\label{Thm-6.5}
 		Under the setting of Theorem \ref{Thm-6.4}, we have that Theorem \ref{Thm-6.3} holds true up to $T = \infty$:
 		\begin{equation}\label{6.36}
 			\BA_{_{F,q}} \in MReg \left( L^p \left( 0, \infty; \bs{W}^q_{\sigma}(\Omega) \right) \right).
 		\end{equation}
 	\end{thm}
	 
	 \begin{proof}[Proof of Theorem \ref{Thm-6.3}]
	 	We return to $\ds \BA_{_{F,q}}$ as given in \eqref{6.26} $\ds \BA_{_{F,q}} = \hat{\BA}_{_{F,q}} + \Pi$, where $\Pi$ is a benign operator regarding the issue of maximal $L^p$-regularity as it involves: the bounded operator $\ds J \in \calL \left( \bs{W}^q_{\sigma}(\Omega), \bls \right) $ in \eqref{6.21}, the bounded operators $\ds \calC_{\gamma} \in \calL(L^q(\Omega); \bls)$ and $\ds \calC_{\theta_e} \in \calL (\bls, L^q(\Omega))$ in \eqref{6.6}, \eqref{6.7}; the operator $A_{o,q}$ which is $\ds A^{\rfrac{1}{2}}_q$-bounded, see \eqref{5.10}; and the operator $B_{o,q}$ which is simply $B^{\rfrac{1}{2}}$-bounded, see \eqref{6.10}. Thus, it suffices (it is equivalent) to show that $\ds \hat{\BA}_{_{F,q}}$ in \eqref{6.27}, \eqref{6.28} has maximal $L^p$-regularity on $\ds \bs{W}^q_{\sigma}(\Omega)$ up to $T < \infty: \ \ds \hat{\BA}_{_{F,q}} \in MReg \left( L^p \left( 0, T; \bs{W}^q_{\sigma}(\Omega) \right) \right)$. We rewrite $\ds \hat{\BA}_{_{F,q}}$ as
	 	
	 	\begin{equation}\label{6.37}
	 		\hat{\BA}_{_{F,q}} \bs{w} = \bbm -A_q \bs{w}_1 = -A_q \bbm I_1 \\ 0  \ebm \bs{w} \\[4mm] -B_q \left( \bbm 0 \\ I_2 \ebm + DF \right) \bs{w} \ebm
	 	\end{equation}
	 	\nin with domain as in \eqref{6.28}. To this end, we cannot apply directly Theorem \ref{Thm-1.2}. Instead, we shall work with the adjoint $\ds \hat{\BA}_{_{F,q}}^*$, as in the proof of Theorem \ref{Thm-1.2}. For $\ds \bs{w} \in \bs{W}^q_{\sigma}(\Omega)$ and $\ds v_2 \in \left[ \calD(B^*_q) \right]'$, we compute the adjoint of $B_qDF:$
	 	\begin{equation}\label{6.38}
	 		\ip{B_qDF \bs{w}}{v_2}_{L^q(\Omega)} = \ip{\bs{w}}{F^*D^*B^*_qv_2}_{\bs{W}^q_{\sigma}(\Omega)}.
	 	\end{equation}
	 	\nin Thus, for $\ds \bbm v_1 \\ v_2 \ebm \in \calD \left( \hat{\BA}_{_{F,q}}^* \right)$, we have
	 	\begin{equation}\label{6.39}
	 		\hat{\BA}_{_{F,q}}^* \bbm v_1 \\ v_2 \ebm = \bbm -A^*_q & 0 \\ 0 & -B^*_q \ebm \bbm v_1 \\ v_2 \ebm + F^*D^*B^*_qv_2.
	 	\end{equation}
	 	\nin By \eqref{6.15c}, we have $\ds D^*B^{*^{\gamma}}_q \in \calL \left( L^{q'}(\Omega), L^{q'}(\Gamma) \right), \ \gamma = \rfrac{1}{2q} - \varepsilon$, and so $\ds F^*D^*B^*_q = \left( F^*D^*B^{*^{\gamma}}_q \right) B^{*^{1-\gamma}}_q$, where $\ds F^* \in \calL \left( L^{q'}(\Gamma), \bs{W}^{q'}_{\sigma}(\Omega) \right)$. Hence, for the perturbation in \eqref{6.38} we estimate
	 	\begin{align}
	 		\norm{F^*D^*B^*_q v_2} &= \norm{\left( F^*D^*B^{*^{\gamma}}_q \right) B^{*^{1-\gamma}}_q v_2} \leq C \norm{ B^{*^{1-\gamma}}_q v_2}\\
	 		&\leq C \left[ \norm{ B^{*^{1-\gamma}}_q v_2} + \norm{ A^{*^{1-\gamma}}_q v_1} \right]\\
	 		&\leq C \norm{\bbm A^*_q & 0 \\ 0 & B^*_q \ebm^{1 - \gamma} \bbm v_1 \\ v_2 \ebm}
	 	\end{align}
	 	\nin and $\ds F^*D^*B^*_q$ is $\ds \bbm A^*_q & 0 \\ 0 & B^*_q \ebm^{1 - \gamma}$-bounded, $1 - \gamma < 1$, where $\bbm A^*_q & 0 \\ 0 & B^*_q \ebm$ has maximal $L^p$-regularity on $\bs{W}^{q'}_{\sigma}(\Omega)$ up to $T < \infty$. We now proceed as in the proof of Theorem \ref{Thm-1.2}, that is, Step 3. By a known perturbation result \cite[Theorem 6.2, p 311]{Dore:2000} or \cite[Remark 1i, p 426 for $\beta = 1$]{KW:2001} we conclude from \eqref{6.39} that $\ds \hat{\BA}^*_{_{F,q}}$- and hence $\ds \BA^*_{_{F,q}}$ in \eqref{6.24} has maximal $L^p$-regularity on $\bs{W}^{q'}_{\sigma}(\Omega)$ up to $T < \infty$. We finally conclude that $\ds \BA_{_{F,q}}$ has maximal $L^p$-regularity in $\bs{W}^q_{\sigma}(\Omega)$ up to $T < \infty$ via Step 4 of the proof of Theorem \ref{Thm-1.2}, as $\bs{W}^q_{\sigma}(\Omega)$ is UMD.
	 \end{proof}
 
 	\begin{proof}[Proof of Theorem \ref{Thm-6.5}] Since, under the setting of Theorem \ref{Thm-6.4}, the s.c. analytic semigroup $\ds e^{\BA_{_{F,q}} t}$ is also uniformly stable on $\bs{W}^q_{\sigma}(\Omega)$, see \eqref{6.35}, then maximal $L^p$-regularity holds up to $T = \infty$. 
 	\end{proof}
	
	\nin \textbf{Acknowledgments} 
	\begin{enumerate}[1.]
		\item The authors wish to thank Giovanni Dore, University of Bologna, for pointing out that a family	of operators is R-bounded if and only if the dual family is R-bounded holds true in UMD spaces,	as given by reference \cite{HNVW:2016}, thus going beyond the special case in \cite{KW:2004} for $L^q(\Omega)$- spaces.
		\item The research of I. L. and R. T. was partially supported by the National Science Foundation
		under grant DMS-1713506. The research of B. P. was partially supported by the ERC advanced
		grant 668998 (OCLOC) under the EU’s H2020 research program.	
	\end{enumerate}

\end{document}